\documentclass[aos,preprint]{imsart}

\RequirePackage{amsthm,amsmath,amsfonts,amssymb}
\usepackage{enumitem}
\RequirePackage[numbers]{natbib}
\RequirePackage[colorlinks,citecolor=blue,urlcolor=blue]{hyperref}
\RequirePackage{graphicx}
\usepackage{graphicx}
\usepackage{caption}

\startlocaldefs
\theoremstyle{plain}

\newtheorem{theorem}{Theorem}[section]
\newtheorem{lemma}[theorem]{Lemma}
\newtheorem{corollary}[theorem]{Corollary}

\newtheorem{Assumption}{Assumption}
\newtheorem{remark}{Remark}[section]
\theoremstyle{remark}

\newcommand{\non}{\nonumber}

\endlocaldefs

\begin{document}
	
	\begin{frontmatter}
		\title{Asymptotic properties of a multicolored random reinforced urn model with an application to multi-armed bandits}

		\begin{aug}
			
			\author[A]{\fnms{Li}~\snm{Yang}\ead[label=e1]{baizd@nenu.edu.cn}},
			\author[A]{\fnms{Jiang}~\snm{Hu}\ead[label=e2]{huj156@nenu.edu.cn}}
			\author[A]{\fnms{Jianghao}~\snm{Li}\ead[label=e3]{lijh458@nenu.edu.cn}}
			\and
			\author[A]{\fnms{Zhidong}~\snm{Bai}\ead[label=e4]{baizd@nenu.edu.cn}}

			\address[A]{KLASMOE and School of Mathematics and Statistics, Northeast Normal University, Changchun, 130024, China\printead[presep={,\ }]{e1,e2,e3,e4}}
		\end{aug}

		\begin{abstract}
			The random self-reinforcement mechanism, characterized by the principle of ``the rich get richer'', has demonstrated significant utility across various domains. One prominent model embodying this mechanism is the random reinforcement urn model. This paper investigates a multicolored, multiple-drawing variant of the random reinforced urn model. We establish the limiting behavior of the normalized urn composition and demonstrate strong convergence upon scaling the counts of each color. Additionally, we derive strong convergence estimators for the reinforcement means, i.e., for	the expectations of the replacement matrix's diagonal elements, and prove their joint asymptotic normality. It is noteworthy that the estimators of the largest reinforcement mean are asymptotically independent of the estimators of the other smaller reinforcement means. Additionally, if a reinforcement mean is not the largest, the estimators of these smaller reinforcement means will also demonstrate asymptotic independence among themselves.	Furthermore, we explore the parallels between the reinforced mechanisms in random reinforced urn models and multi-armed bandits, addressing hypothesis testing for expected payoffs in the latter context.
			
		\end{abstract}


		
	\end{frontmatter}

	\section{Introduction}\label{sec1}
	The urn model, originally introduced by \cite{r1}, is described as follows. An urn initially contains $b$ black balls and $w$ white balls. Each stage involves randomly drawing a ball, with the probability of drawing a black ball being equal to the fraction of black balls in the urn. Upon drawing, the ball is returned to the urn along with an additional $a$ balls of the same color. This process defines the P\'{o}lya urn model, which has been extensively employed in medical simulations of disease infections. Over the past century, various extensions of the P\'{o}lya urn model have been developed, leading to its application across diverse fields such as economics (\cite{r32,r4}), biology (\cite{r38,r39,r40}), informatics (\cite{r41,r43,r44}), machine learning (\cite{r33,r34}), and adaptive designs in clinical trials (\cite{r35,r36,r37}).
	
	Model variations primarily stem from differing protocols for the addition and drawing of balls. Bernard Friedman expanded on the P\'{o}lya urn by introducing a model in \cite{r45} where each drawn ball is returned to the urn with an additional $a$ balls of the same color and $b$ balls of a different color. Subsequently, in \cite{r35}, Wei and Durham adapted Friedman's model for use in adaptive designs for clinical trials. Moreover, this model was further generalized to the GFU model (Generalized Friedman's Urn model, also called GPU in the literature), which incorporates multicolored balls and an irreducible replacement matrix. Each ball color in the GFU model represents a different treatment option, administered sequentially to patients as they are received. Upon each patient's arrival, a ball is drawn at random, and the patient is assigned the treatment corresponding to the ball's color. Treatment assignments are then adjusted based on the patient's response, effectively altering the probability of each treatment being assigned subsequently. Under specific conditions, such as finite moments for the irreducible replacement matrix, the normalized urn composition converges almost surely to the dominant eigenvector of the conditional expectation of the replacement matrix, detailed further in \cite{r21,r22,r23,r24,r25,r26}. Recent years have seen the exploration of numerous other generalizations, including urn models with triangular (\cite{r17}), group multiplication (\cite{r49, r31}), non-square (\cite{r29, r30}), and anti-diagonal (\cite{r8, r9}) replacement matrices, as well as Randomly Reinforced P\'{o}lya urn (RRU) models featuring diagonal replacement matrices.
	
	The foundational P\'{o}lya urn model underpins the RRU models, where the rule for adding balls post-drawing is captured by the replacement matrix
	\begin{equation*}
		{\bf R}_n=\begin{pmatrix}
			a~ & 0\\
			0~ & a
		\end{pmatrix},
	\end{equation*}
	where the $k$th row ($k=1,2$) of the matrix denotes the addition of $a$ balls of the color corresponding to the drawn ball. Research in \cite{r2} confirms that the proportion of black balls converges almost surely to a random variable $Y$, following a Beta distribution $Y \sim{\rm Beta}(b/a, w/a)$ (refer to Theorem 1 in section 9.1). An extension to the two-color generalized P\'{o}lya urn model is considered in \cite{r3}, where a random number $A_n$ of balls is added at each stage, described by the replacement matrix
	\begin{equation*}
		{\bf R}_n=\begin{pmatrix}
			A_n & 0\\
			0 & A_n
		\end{pmatrix},
	\end{equation*}
	with $\{A_n, n=1,2,\dots\}$ as a sequence of independent and identically distributed(i.i.d.) non-negative random variables. Following $n$ stages, the urn's composition and the count of each color drawn are detailed as follows,
	\begin{eqnarray*}
		H_{n1}&=&b+\sum\limits_{i=1}^{n}X_iA_i, \ \ H_{n2}=w+\sum\limits_{i=1}^{n}(1-X_i)A_i,\\
		Z_n&=&\frac{H_{n1}}{H_{n1}+H_{n2}},\ \ N_{n1}=\sum\limits_{i=1}^{n}X_i, \ \ N_{n2}=\sum\limits_{i=1}^{n}(1-X_i),
	\end{eqnarray*}
	where $X_i$ denotes the result of the $i$th draw, and $X_i\sim {\rm Bernoulli}(Z_{i-1})$. The paper \cite{r3} shows that the proportion of black balls, $Z_n$, behaves as a bounded martingale and thus converges almost surely to a limit $Z_\infty$. Likewise, the proportion of black balls drawn, $N_{n1}/n$, also converges almost surely to $Z_\infty$. When $A_n$ is uniformly distributed at point $a$, this model reverts to the standard P\'{o}lya urn model, hence $Z_\infty\sim{\rm Beta}(b/a, w/a)$. When $A_n$ follows a two-point distribution, the paper \cite{r3} derived a random variable $G_n$ with respect to $H_{n1}$ and $H_{n2}$ such that the distributional distance between $G_n$ and $Z_\infty$ approaches zero.  In a later study, the paper \cite{r5} defined a two-color RRU model with a distinct replacement rule for each color, formalized as
	\begin{equation}\label{eqx1}
		{\bf R}_n=\begin{pmatrix}
			A_{n1} & 0\\
			0 & A_{n2}
		\end{pmatrix},
	\end{equation}
	where $\{A_{n1}, n=1,2,\dots\}$ and $\{A_{n2}, n=1,2,\dots\}$ are sequences of i.i.d. non-negative random variables with bounded support, respectively. The paper \cite{r5} demonstrated that the proportion of black balls, $Z_n$, eventually behaves as a bounded super- or sub-martingale, hence converging almost surely to a limit $Z_\infty$. Assuming that
	\begin{equation*}
		m_1={\mathbb E}A_{n1}>{\mathbb E}A_{n2}=m_2,
	\end{equation*}
	it was proven that $\{Z_\infty=1\}$ holds almost surely. Additionally, a multicolored RRU model was considered in \cite{r5}, with the replacement matrix defined as
	\begin{equation}\label{eqx2}
		{\mathbf R}_n=\begin{pmatrix}
			A_{n1} & 0 & \dots & 0\\
			0 & A_{n2} & \dots & 0\\
			\vdots & \vdots & \ddots & \vdots\\
			0 & 0 & \dots & A_{nd}
		\end{pmatrix},
	\end{equation}
	and it was established that under the condition
	\begin{equation*}
		m_1={\mathbb E}A_{n1}>{\mathbb E}A_{nk}=m_k,\ \ \mbox{for\ all}\ k\in\{2,3,\dots,d\},
	\end{equation*}
	the proportion of the first color of balls, $Z_n$, converges to 1 almost surely, thus relegating the proportions of other colors to zero. Correspondingly, the proportions of other colors drawn, $N_{nk}/n(k\in\{2,3,\dots,d\})$, also converge to zero almost surely.
	
	When implemented in adaptive clinical trial designs, the RRU model ensures that the proportion of patients receiving the superior treatment, $N_{n1}/n$, converges almost surely to 1, whereas the proportion receiving the lesser treatment, $N_{n2}/n$, approaches zero. If treatments are deemed equivalent, the distribution of patients across treatments stabilizes to a non-degenerate random variable ranging between [0,1]. For the two-color RRU model represented by the replacement matrix (\ref{eqx1}), the paper \cite{r10} initially examined the precise order of the number of patients allocated to the lesser treatment, demonstrating that for case $m_1 \ge m_2$, there exists a random variable $\eta^2$ with $\mathbb P(0<\eta^2<\infty)=1$ such that
	\begin{equation*}
		\lim\limits_{n\to\infty}\frac{N_{n2}}{n^{m_2/m_1}}\overset{a.s.}{\to}\eta^2.
	\end{equation*}
	Moreover, they established joint asymptotic normality for the strong consistency estimators of parameters $m_1,m_2$. Building on this, they addressed the hypothesis testing problem
	$$H_0: m_1=m_2\ \ \ {\rm versus}\ \ \ H_1:m_1>m_2,$$
	and demonstrated that the test's power is a function of $\eta$. A few years later, for the multicolored RRU model with replacement matrix (\ref{eqx2}), the paper \cite{r11} investigated the exact convergence rate of the number of patients allocated to inferior treatments. Under the $L\log L$ moment assumption for the replacement matrix, it was found that for all $k,s\in\{1,2,\dots,d\}$, there exist $d$ positive random variables $\xi_{k}$ such that:
	\begin{equation*}
		\frac{H^{1/m_k}_{nk}}{H^{1/m_s}_{ns}}\overset{a.s.}{\to}\frac{\xi^{1/m_k}_k}{\xi^{1/m_s}_s},\ \ \ 	\frac{N^{1/m_k}_{nk}}{N^{1/m_s}_{ns}}\overset{a.s.}{\to}\frac{(\xi_k/m_k)^{1/m_k}}{(\xi_s/m_s)^{1/m_s}},
	\end{equation*}
	and
	\begin{equation*} \frac{H_{nk}}{n^{m_k/m_{\max}}}\overset{a.s.}{\to}\left(\frac{m_{\max}}{\sum_{\{j:m_j=m_{\max}\}}\xi_j}\right)^{m_k/m_{\max}}\xi_k,
	\end{equation*}
	where $H_{nk}$ denotes the number of balls of color $k$ in the urn after $n$ stages and $N_{nk}$ denotes the number of times type $k$ balls have been drawn in the preceding $n$ stages. In scenarios where both $H_{0k}$ and $A_{nk}$ are integers, an alternate formulation for $\xi_k(k\in\{1,2,\dots,d\})$ was proposed, with the conjecture that this result extends to non-integer systems as well. This marked the first precise determination of the limiting distribution of $H_{nk}$, i.e., the distribution of $Z_\infty$, under an RRU model with non-uniquely maximal expected values for the diagonal entries of the replacement matrix. For further details on the second order asymptotic properties of $H_{nk}$, refer to \cite{r15,r16}. Additionally, the studies in \cite{r19} and \cite{r20} explored a response adaptive design, defined via a two-color urn model aiming at fixed asymptotic allocations, known respectively as MRRU and ARRU in the literature.
	
	Recent advancements include a two-color RRU model with multiple draws per stage, as proposed in \cite{r12}. At each stage, rather than drawing a single ball, a random number $N_n$ of balls are drawn. Define $X_n$ as the number of first-color balls among the $N_n$ drawn. Following $n$ stages, the urn's composition and the count of each color drawn are detailed as follows,
	\begin{eqnarray*}
		H_{n1}=H_{0,1}+\sum\limits_{i=1}^{n}X_iA_{i1},\ \ \	H_{n2}=H_{0,2}+\sum\limits_{i=1}^{n}(N_i-X_i)A_{i2},\ \ \
		Z_n=\frac{H_{n1}}{H_{n1}+H_{n2}},
	\end{eqnarray*}
	and
	\begin{eqnarray*}
		N_{n1}=\sum\limits_{i=1}^{n}X_i,\ \ \
		N_{n2}=\sum\limits_{i=1}^{n}(N_i-X_i),
	\end{eqnarray*}
	respectively. The paper \cite{r12} elucidated the first- and second-order asymptotic properties of $Z_n$. Subsequently, the paper \cite{r13} investigated the precise orders of $N_{n1}$ and $N_{n2}$. Additionally, they developed strongly consistent estimators for $m_1,m_2$, and established the joint asymptotic normality of these estimators. Based on these findings, they examined the one-sided hypothesis testing problem concerning $m_1=m_2$.
	
	This paper focuses on a multicolored RRU model with multiple drawing (MMRU model). Specifically, we extend the findings from \cite{r10}, \cite{r12}, and \cite{r13} to a multicolored, multiple drawing RRU model under a more lenient moment condition for the replacement matrix. Initially, under the MMRU model, we determine the limit of the normalized urn composition. As the number of updates $n$ progresses towards infinity, the proportion of each type of ball converges almost surely. Our results indicate that the number of balls with the highest reinforcement expectation increases indefinitely in the order of $n$, while the counts of other types diminish to $o(n)$. Additionally, we establish the exact order for the count of balls of each type in the urn and of each type drawn, along with their respective convergences post-scaling. The reinforcement mechanism within the MMRU model aligns with a branch of reinforcement learning known as the multi-armed bandits problem, where the reinforcement expectations mirror the expected payoffs of the bandits' arms. Consequently, under the MMRU model, we derive strong consistency estimators for the reinforcement expectations and ascertain the joint asymptotic normality of these estimators. Finally, we address the hypothesis testing problem concerning the expected payoffs for the arms of the multi-armed bandits.
	
	The MMRU model can also be related to iterative processes with enhancement mechanisms, exemplified by the election of candidates. Consider an organization with $d$ candidates, where initially, only a subset of eligible voters cast their votes, possibly due to insufficient knowledge regarding the candidates. Each candidate starts with a certain number of supporters. At each subsequent stage $n$, $N_n$ previously decisive voters are randomly selected, and these individuals may influence yet undecided voters to support their chosen candidate. The persuasiveness of a previously decisive voter largely depends on the ``attractiveness'' of the candidate they support, with each supporter capable of swaying an equivalent number of undecided voters within a given timeframe, averaged as the ``attractiveness factor'' of the candidate. This enhancement process is modeled through the randomness of $A_{nk}(n\ge1,\ k\in\{1,2,\dots,d\})$, with candidates possessing higher ``attractiveness factors'' seemingly holding an advantage.
	
	The paper is organized as follows.  Section \ref{sec2} introduces the necessary notation and assumptions. Section \ref{sec3} presents the main results. In Section \ref{sec4}, we discusses the derivation of strong consistency estimators for the reinforcement means and the joint asymptotic normality of these estimators. Section \ref{sec5} outlines the testing methodology for the reinforcement means. Section \ref{sec6} explores practical examples and displays simulation outcomes. Finally, Appendices \ref{secA0}, \ref{secA} and \ref{secB} provide proofs of main results, technical results and recollections of auxiliary findings, respectively.

	\section{Notation and Assumptions}\label{sec2}
	We initiate our model with an urn containing balls of $d$ distinct colors. Define the initial composition of the urn by ${\bf H}_{0}=(H_{01},H_{02},\dots,H_{0d})^{\top}$, where $H_{0k}$ denotes the count of the $k$th color balls, for $k \in \{1,2,\dots,d\}$. Let $S_0=|{\bf H}_0|=\sum_{k=1}^{d}H_{0k}$ represent the total number of balls initially. At each stage $n$, a random number $N_n$ of balls is drawn, either with or without replacement, depending on the experiment design. The composition of these drawn balls is represented by ${\bf X}_n=(X_{n1},X_{n2},\dots,X_{nd})^{\top}$, where $X_{nk}$ indicates the number of balls of the $k$th color drawn at stage $n$. The replacement matrix is specified as
	\begin{equation*}
		{\bf R}_n=\begin{pmatrix}
			A_{n1}&0&\dots&0\\
			0&A_{n2}&\dots&0\\
			\vdots&\vdots&\ddots&\vdots\\
			0&0&\dots&A_{nd}
		\end{pmatrix},
	\end{equation*}
	where each $A_{n1}, A_{n2},\dots,A_{nd}$ are non-negative random variables representing the number of balls of each color added back to the urn after drawing. The urn's composition after $n$ stages becomes ${\bf H}_n=(H_{n1},H_{n2},\dots,H_{nd})^{\top}$, and the total number of balls is updated to $S_n=|{\bf H}_n|=\sum_{k=1}^d H_{nk}$. The renewal process for the urn can be expressed as
	$${\bf H}_n={\bf H}_0+\sum\limits_{i=1}^{n}{\bf X}_i{\bf R}_i,\ \ \  S_n=S_0+\sum\limits_{i=1}^{n}\sum\limits_{k=1}^dA_{ik}X_{ik}.$$
	The normalized urn composition is then defined as
	$${\bf Z}_{n}=(Z_{n1},Z_{n2},\dots,Z_{nd})^{\top}=\frac{{\bf H}_n}{S_n}.$$
	For each color $k$, the cumulative number of balls of that color drawn by stage $n$ is
	$$N_{A_k,n}=\sum_{j=1}^{n}X_{jk}.$$Furthermore, we set 
	\begin{eqnarray*}
		\mathcal{F}_0&=&\{\varnothing,\omega\},\ \ \mathcal{F}_n=\sigma\{N_1,{\bf X}_1,{\bf R}_1,\dots,N_n,{\bf X}_n,{\bf R}_n\},\\
		\mathcal{G}_n&=&\sigma\{N_1,{\bf X}_1,{\bf R}_1,\dots,N_n,{\bf X}_n,{\bf R}_n, N_{n+1}\}.
	\end{eqnarray*}
	We assume that the number of balls drawn at each stage, the drawing rule and the replacement matrix satisfy the following distributions.
	\begin{enumerate}[label=\Roman*.]
		\item Support of $N_{n+1}$: The support of $N_{n+1}$ is contained in $\{1\wedge C_1,2\wedge C_1,\dots,S_{n}\wedge C_1\}$, where $C_1$ is some constant.
		\item Conditional distribution of ${\bf X}_{n+1}$ given $\mathcal{G}_n$:\\
		{\rm (a)} Without replacement. Conditional on $\mathcal{G}_n$, 
		$${\bf X}_{n+1}=(X_{n+1,1},X_{n+1,2},\dots,X_{n+1,d})^{\top}\sim {\rm Multi-Hyper}(S_n,N_{n+1},{\bf H}_n),$$
		that is, set ${\bf{\cal X}}_{n+1}=\{(x_1,x_2\dots,x_d)^{\top}: \prod_{k=1}^{d}\mathbb I(0\le x_k\le H_{nk},\sum_{k=1}^{d}x_k=N_{n+1})=1\}$. Given ${\cal G}_n$, for all ${\bf x}_{n+1}\in {\bf{\cal X}}_{n+1}$, we have
		$${\mathbb P}_{{\bf x}_{n+1}}={\mathbb P}({\bf X}_{n+1}={\bf x}_{n+1})={\mathbb P}(X_{n+1,1}=x_{1},\dots,X_{n+1,d}=x_{d})=\frac{ \prod_ {k=1}^{d}{H_{nk}\choose{x_{k}}}}{{S_n\choose N_{n+1}}}.$$
		{\rm (b)} With replacement. Conditional on $\mathcal{G}_n$, 
		$${\bf X}_{n+1}=(X_{n+1,1},X_{n+1,2},\dots,X_{n+1,d})^{\top}\sim {\rm Multinomial}(N_{n+1},{\bf Z}_n),$$
		that is, set ${\bf{\cal X}}_{n+1}=\{(x_1,\dots,x_d)^{\top}: \prod_{k=1}^{d}\mathbb I(0\le x_k\le N_{n+1},\sum_{k=1}^{d}x_k=N_{n+1})=1\}$. Given ${\cal G}_n$, for all ${\bf x}_{n+1}\in {\cal X}_{n+1}$, we have
		$${\mathbb P}_{{\bf x}_{n+1}}={\mathbb P}({\bf X}_{n+1}={\bf x}_{n+1})={\mathbb P}(X_{n+1,1}=x_{1},\dots,X_{n+1,d}=x_{d})=\frac{N_{n+1}!}{ \prod_ {k=1}^ {d} x_{k}!} \prod_ {k=1}^ {d} Z_{nk}^{x_{k}}.$$
		\item Replacement matrix ${\bf R}_{n+1}$: For all $k\in\{1,2,\dots,d\}$, $A_{n+1,k}$ takes value on $[1,+\infty)$, and ${\bf R}_{n+1}$ is independent of ${\cal G}_n$. However, the components of ${\bf R}_{n+1}$ may be correlated with each other.
	\end{enumerate}
	Subsequent assumptions are necessary for the analysis:
	\begin{Assumption} \label{as1}
		There exists a constant $C_2$ such that for all $k\in\{1,2,\dots,d\}$,
		$\sup_n\mathbb E(A^3_{nk})<C_2.$
	\end{Assumption}

	\begin{Assumption} \label{as2}
		Assume that for all $k\in\{1,2,\dots,d\}$,
		\begin{equation*}
			m_{nk}:=\mathbb E(A_{nk})\to m_k,\ \ when\ n\to\infty, 
		\end{equation*}
		where $m_1,m_2,\dots,m_d$ are $d$ finite real numbers satisfying 
		\begin{equation*}
			m_1=\dots=m_t>m_{t+1}\ge m_{t+2}\ge\dots\ge m_d
		\end{equation*}
		for some $t\in\{1,2,\dots,d\}$.
		In addition, we assume that for all $n=1,2,\dots$, 
		\begin{equation*}
			m_{n1}=m_{n2}=\dots=m_{nt}.
		\end{equation*}
	\end{Assumption}
	
	\begin{Assumption} \label{as3}
		Assume that for all $k,j\in\{1,2,\dots,d\}$, 
		\begin{equation*}
			\sum\limits_{n=1}^{\infty}\frac{|m_{n+1,k}m_j-m_km_{n+1,j}|}{n}<\infty.
		\end{equation*}
	\end{Assumption}

	\begin{Assumption} \label{as4}
		Assume that for all $k\in\{1,2,\dots,d\}$,
		\begin{equation*}
			|m_{nk}-m_k|=o(n^{-m_k/2m_1}).
		\end{equation*}
	\end{Assumption}

	\begin{remark}
		In the literature, the replacement matrices $\{{\bf R}_n,n\ge 1\}$ are often assumed to be i.i.d. as in \cite{r10} and uniformly bounded as detailed in \cite{r12} and \cite{r13}. However, Assumption \ref{as1} introduces a less restrictive condition by only necessitating control over the second-order moments for the convergence analysis of the normalized urn composition $Z_{nk}$. When delving into the precise order of $Z_{nk}$, the requirement extends to the third-order moments. Importantly, this assumption can be further relaxed to the $2+\epsilon$ order moments when the matrix components $A_{nk}$ are uncorrelated. This nuanced understanding of moment conditions enriches the theoretical analysis, facilitating a broader application of the results. Detailed proofs of these assertions can be found in the proof of Theorem \ref{th4} and Lemma \ref{leA2}, where these moment conditions are leveraged to establish the desired asymptotic properties.
		Further, to ascertain the exact order of $Z_{nk}$, it is necessary to monitor the convergence rate of $m_{nk}$ towards $m_k$. If there exists $\epsilon > 0$ such that $|m_{nk} - m_k| = O((\log n)^{-(1+\epsilon)})$, then Assumption \ref{as3} is satisfied. Moreover, to determine the convergence rate for the estimator of $m_k$, a more rapid convergence of $m_{nk}$ is requisite, as specified in Assumption \ref{as4}.
	\end{remark}

	In the sequel, the paper presents results and proofs under the scenario of case II(a), where the drawing rule adheres to a Multidimensional Hypergeometric Distribution. It should be noted that these results remain valid under case II(b), where the drawing rule follows a Multinomial Distribution. Under certain assumptions, this work details the first-order asymptotic properties of the normalized urn composition, the exact order of each color of balls, the tally of each color of balls drawn, and provides strongly consistent estimations of the limit of $\mathbb E({\bf R}_n)$ along with the corresponding convergence rates. Both the first-order asymptotic properties and the exact order calculations involve the conditional expectations of $\{{\bf X}_n, n=1,2,\dots\}$, while estimations of the limit of $\mathbb E({\bf R}_n)$ and rate of convergence for $\mathbb E({\bf R}_{n})$ utilize the conditional covariance matrix of $\{{\bf X}_n, n=1,2,\dots\}$. Notably, for ${\bf X}^{(1)}_{n+1} \sim \text{Multi-Hyper}(S_n, N_{n+1}, {\bf H}_n)$ and ${\bf X}^{(2)}_{n+1} \sim \text{Multinomial}(N_{n+1}, {\bf Z}_n)$, the mathematical expectations and covariance matrices of ${\bf X}^{(1)}_{n+1}$ and ${\bf X}^{(2)}_{n+1}$ converge as $n$ increases.

	\section{Asymptotic properties of ${\bf Z}_n$}\label{sec3}
	In this section, we consider the limit of the normalized urn composition ${\bf Z}_n$ and the exact order of its components. 
	\subsection{The limit of ${\bf Z}_n$}
	In order to obtain the limit of ${\bf Z}_n$, we first need to consider whether the number of balls of each color tends to infinity as $n$ tends to infinity. 
	\begin{lemma}\label{le1}
		Under Assumption \ref{as1}, for all $k\in\{1,2,\dots,d\}$, we have
		$H_{nk}\overset{a.s.}{\to}\infty.$
	\end{lemma}
	
	The Lemma \ref{le1} shows that as $n$ tends to infinity, the number of balls of each color in the urn converges almost surely to  infinity eventually. Next, we consider the limit of ${\bf Z}_n$.
	\begin{theorem}\label{th1}
		Under Assumptions \ref{as1}-\ref{as2}, we have, for all $k\in\{1,\dots,t\}$, $Z_{nk}$ is eventually a bounded submartingale. Therefore, it converges almost surely to a random variable $Z_k$ in $[0,1]$.
	\end{theorem}
	
	Theorem \ref{th1} shows the limit of the proportions of balls of the first $t$ colors in the urn. In the following, we consider the limit of the proportion of balls of the latter $d-t$ colors.
	\begin{theorem}\label{th2}
		Under Assumptions \ref{as1}--\ref{as2}, we have
		$$\sum\limits_{k=1}^{t}Z_{nk}\overset{a.s.}{\to}1.$$
		As a result, for all $k\in\{t+1,\dots,d\}$, $Z_{nk}$ converges almost surely to $Z_k$ with $\mathbb P(Z_k=0)=1$.
	\end{theorem}

	\begin{corollary}\label{co1}
		Under Assumptions \ref{as1}--\ref{as2} and the condition
		\begin{equation}\label{eqn3}
			\mathbb E(N_{n+1}|{\cal F}_n)\overset{a.s.}{\to}N,
		\end{equation}
		we have
		\begin{equation*}
			\frac{H_{nk}}{n}\overset{a.s.}{\to}\begin{cases}
				Nm_1Z_k,\ &\text{if\ $k\in\{1,\dots,t\},$}\\
				0, \ &\text{if\ $k\in\{t+1,\dots,d\},$}\\
			\end{cases}\ \ \ and \ \ \ \frac{S_n}{n}\overset{a.s.}{\to}Nm_1.
		\end{equation*}
	\end{corollary}

	\subsection{The exact order of $Z_{nk}(k\in\{1,2,\dots,d\})$}
	For all $k\in\{t+1,\dots,d\}$, in Corollary \ref{co1}, we obtain that $H_{nk}/n\overset{a.s.}{\to}0$, and in Lemma \ref{leA3}, for all $0<\gamma_k<m_k/m_1$, $H^{-1}_{nk}=O_{a.s.}(n^{-\gamma_k})$. In the following, we consider the exact order of $H_{nk}$.
	\begin{theorem}\label{th3}
		Under Assumptions \ref{as1}--\ref{as3} and condition (\ref{eqn3}), for all $k\in\{1,2,\dots,d\}$, $j\in\{1,\dots,k-1,k+1,\dots,d\}$, there exists a random variable $\xi_{kj}$ on $(0,\infty)$ such that 
		$$\frac{H_{nk}}{H^{m_k/m_j}_{nj}}\overset{a.s.}{\to}\xi_{kj}.$$
	\end{theorem}

	In Theorem \ref{th3}, we establish a relationship concerning the order of two different types of balls in the urn as their quantities approach infinity. However, to precisely determine the order of balls for each type, it is essential to know the exact order of at least one type of ball. The subsequent corollary specifies the  order for the first $t$ types of balls.
	\begin{corollary}\label{co2}
		Under Assumptions \ref{as1}--\ref{as3} and condition (\ref{eqn3}), when $t=1$, then $Z_1$ equals to 1 almost surely, and when $t>1$, for all $k\in\{1,2,\dots,t\}$, we have ${\mathbb P}(Z_k\in(0,1))=1$.
	\end{corollary}
	
	In the following, for all $k\in\{t+1,\dots,d\}$, we show the exact order of $Z_{nk}$ and $H_{nk}$. In addition, we consider the exact order of balls of each color drawn, i.e., $N_{A_k,n}$(recall the definition in Section \ref{sec2}). In \cite{r10}, for a two-color RRU model with one ball randomly drawn at each stage, May and Flournoy show that the limit of balls of each color drawn is related to the power of the test problem $H_0:m_1=m_2$ versus $H_1:m_1>m_2$.
	
	\begin{theorem}\label{th6}
		Under Assumptions \ref{as1}--\ref{as3} and condition (\ref{eqn3}), we have that 
		\begin{enumerate}[label=(\roman*)]
			\item  For all $j\in\{t+1,\dots,d\}$, there exists a random variable $\tilde{Z}_j$ such that
			$$n^{1-m_j/m_1}Z_{nj}\overset{a.s.}{\to}\tilde{Z}_j,$$
			where $\tilde{Z}_j\overset{a.s.}{=}\frac{\xi_{j1}Z^{m_j/m_1}_1}{\left(Nm_1\right)^{1-m_j/m_1}}$, and $\mathbb P(\tilde{Z}_j\in(0,+\infty))=1$.
			\item For all $j\in\{t+1,\dots,d\}$, we have
			$$\frac{H_{nj}}{n^{m_j/m_1}}\overset{a.s.}{\to}(Nm_1Z_1)^{m_j/m_1}\xi_{j1}.$$
			\item For all $k\in\{1,2,\dots,d\}$, we have
			\begin{eqnarray}\label{eqn10}
				\frac{N_{A_k,n}}{n^{m_k/m_1}}=\frac{1}{n^{m_k/m_1}}\sum\limits_{j=1}^{n}X_{jk}\overset{a.s.}{\to}\begin{cases}
					NZ_k, &\text{$k\in\{1,\dots,t\}$},\\
					\frac{m_1}{m_k}N\tilde{Z}_k, &\text{$k\in\{t+1,\dots,d\}$}.
				\end{cases}
			\end{eqnarray}
		\end{enumerate}
	\end{theorem}

	\section{Parameter Estimation}\label{sec4}
	
	In the previous discussion (Section \ref{sec3}), we detailed the specific order and asymptotic limits of each color type of balls within the urn model. However, the actual distributional forms of these asymptotic limits were not determined. This omission precludes the direct application of these results to hypothesis testing for evaluating the expected payoffs of the arms in a multi-armed bandits scenario. Therefore, to adequately equip our analysis for addressing such hypothesis tests, it is essential first to establish and validate the estimators for the unknown parameters within our model, as well as to assess the convergence rates of these estimators.
	
	This section presents the development and evaluation of these estimators. The focus is to provide robust statistical tools that can help infer the parameters governing the dynamics of the urn, which are critical for comprehending the long-term behavior of the system and for testing hypotheses related to the operational efficacy of multi-armed bandits algorithms. This work is essential for extending our theoretical results to practical applications in statistical decision-making and optimization in stochastic environments.

	\subsection{Estimators of the parameters} In this section, we shows the results of the estimators of $N$, $m_k(k\in\{1,2,\dots,d\})$ and $Z_k(k\in\{1,\dots,d\})$. Further, we obtain the joint convergence rate of the estimators of $m_k$.
	\begin{theorem}\label{th4}
		Assume that Assumptions \ref{as1}--\ref{as3} and condition (\ref{eqn3}) hold. Further, assume that
		\begin{equation}\label{eqn8}
			\mathbb{E}(N^{2}_{n+1}|\mathcal{F}_n)\to Q,\ \ a.s..
		\end{equation}
		Then, we have
		\begin{equation*}
			\hat{\mu}_{n}=\frac{\sum\limits_{j=1}^{n}N_j}{n}\overset{a.s.}{\to}N,\ \ \hat{q}_{N,n}=\frac{\sum\limits_{j=1}^{n}N^{2}_j}{n}\overset{a.s.}{\to}Q,\ \ \hat\nu_{nk}=\frac{1}{n}\sum\limits_{j=1}^n\frac{X_{jk}}{N_j}\overset{a.s.}{\to}Z_k(k\in\{1,2,\dots,d\}).
		\end{equation*}
		And, the random variable 
		\begin{equation*}
			\hat{m}_{A_k,n}=\frac{\sum\limits_{j=1}^nA_{jk}X_{jk}}{N_{A_k,n}}
		\end{equation*}
		is a strongly consistent estimator of $m_k(k\in\{1,2,\dots,d\})$. Further, assume that 
		\begin{equation}\label{eqn9}
			q_{A_k,n}=\mathbb{E}(A^2_{nk})\to q_k,\ \ q_{A_kA_s,n}=\mathbb E(A_{nk}A_{ns})\to q_{ks},
		\end{equation}
		then for $k\in\{1,2,\dots,d\}$, we have
		$$\hat{q}_{A_k,n}=\frac{\sum\limits_{j=1}^{n}A^2_{jk}X_{jk}}{N_{A_k,n}}\overset{a.s.}{\to}q_k;$$
		for $k,s$ such that $m_k+m_s>m_1$, on the set $\{Q>N\}=\{N>1\}$,
		$$ \hat{q}_{A_kA_s,n}=\frac{\sum_{j=1}^{n}A_{jk}A_{js}X_{jk}X_{js}}{\sum_{j=1}^{n}X_{jk}X_{js}}\overset{a.s.}{\to}q_{ks};$$
		and for $k,s$ such that $m_k+m_s\le m_1$,
		$$\tilde{q}_{A_kA_s,n}=\frac{\sum_{j=1}^nA_{jk}A_{js}X_{jk}}{N_{A_k,n}}\overset{a.s.}{\to}q_{ks}.$$
	\end{theorem}

	In the following, we show the joint asymptotic normality of the estimators $\hat m_{A_k,n}$.
	\begin{theorem}\label{thm5}
		Assume that Assumptions \ref{as1}-\ref{as4} and the conditions (\ref{eqn3}), (\ref{eqn8}) and (\ref{eqn9}) holds. Define
		$$\sigma^{2}_k=q_k-m^{2}_k\ \ and\ \  c_{ij}=q_{k_i,k_j}-m_{k_i}m_{k_j}.$$		
		Then we have
		$$\left(\sqrt{N_{A_1,n}}(\hat{m}_{A_1,n}-m_1),\dots,\sqrt{N_{A_d,n}}(\hat{m}_{A_d,n}-m_d)\right)^{\top}\overset{stable}\to \mathcal{N}(0,\Sigma),$$
		where
		\begin{equation*}
			\Sigma=    \footnotesize{
				\begin{pmatrix}
					\sigma^{2}_1[Z_1(Q/N-1)+1]&\dots&c_{1t}\sqrt{Z_1Z_t}(Q/N-1)&0&\dots&0\\
					\vdots&\ddots&\vdots&\vdots&\ddots&\vdots\\
					c_{1t}\sqrt{Z_1Z_t}(Q/N-1)&\dots&\sigma^{2}_t[Z_t(Q/N-1)+1]&0&\dots&0\\
					0&\dots&0&\sigma^{2}_{t+1}&\dots&0\\
					\vdots&\ddots&\vdots&\vdots&\ddots&\vdots\\
					0&\dots&0&0&\dots&\sigma^{2}_{d}
			\end{pmatrix}}
		\end{equation*}
		On the set $\{Q=N\}=\{N=1\}$, $\Sigma$ is a diagonal matrix and the diagonal elements are the asymptotic variances of $A_{nk}(k\in\{1,2,\dots,d\})$.
	\end{theorem}
	\begin{remark}
		We note that when $N=1$ and $d=2$, The joint asymptotic covariance matrix of the estimators of the reinforced expectations is consistent with the form of Theorem 4.1 of \cite{r10}.
	\end{remark}
	\begin{remark}
	From Theorem \ref{thm5}, we can infer that the joint asymptotic covariance matrices of $\hat{m}{A_1,n}, \hat{m}{A_2,n}, \dots, \hat{m}_{A_d,n}$ demonstrate a block structure. The dimensions preceding $t$ and those following $t+1$ are asymptotically independent, and the dimensions following $t+1$ are also asymptotically independent among themselves.
	\end{remark}

	\begin{corollary} \label{co4}
		Assume that Assumptions \ref{as1}-\ref{as4} and the conditions (\ref{eqn3}), (\ref{eqn8}) and (\ref{eqn9}) hold. We have that for all $k\in\{1,2,\dots,d-1\}$,
		\begin{equation}\label{eqn36}
			\frac{1}{\sqrt{T^{k,k+1}_n}}\frac{\hat{m}_{A_k,n}-\hat{m}_{A_{k+1},n}-(m_k-m_{k+1})}{\sqrt{\sigma^2_k/N_{A_k,n}+\sigma^2_{k+1}/N_{A_{k+1}}}}\overset{L}{\to}{\cal N}(0,1),		
		\end{equation}
		where
		\begin{eqnarray*}
			T^{k,k+1}_n=\begin{cases}
				T^{k,k+1}_{n,0},\ &\text{$k+1\le t,$}\\
				T^{k,k+1}_{n,1}=1,\ &\text{$k\ge t.$}
			\end{cases}
		\end{eqnarray*}
		and 
		\begin{align*}
			&T^{k,k+1}_{n,0}=\frac{1}{\sigma^2_kN_{A_{k+1},n}+\sigma^2_{k+1}N_{A_k,n}}\Big\{\sigma^2_k[Z_k(Q/N-1)+1]N_{A_{k+1},n}\Big.\\
			&\Big.-2c_{k,k+1}\sqrt{Z_kZ_{k+1}}(Q/N-1)\sqrt{N_{A_k,n}N_{A_{k+1},n}}+\sigma^2_{k+1}[Z_{k+1}(Q/N-1)+1]N_{A_k,n}\Big\}.
		\end{align*}
		And, the convergence of (\ref{eqn36}) still holds if $N$, $Q$, $Z_k(k=1,\dots,t)$, $c_{pq}(1\le p<q\le t)$  and $\sigma^2_k(k=1,\dots,d)$ are replaced by their strongly consistency estimators $\hat{\mu}_n$, $\hat{q}_{N,n}$, $\hat\nu_{nk}$, $\hat{q}_{A_pA_q,n}-\hat{m}_{A_p}\hat{m}_{A_q}$ and $\hat{q}_{A_k,n}-\hat{m}^2_{A_k,n}$, respectively.
	\end{corollary}

	\begin{corollary}\label{co5}
		Assume that Assumptions \ref{as1}-\ref{as4} and the conditions (\ref{eqn3}), (\ref{eqn8}) and (\ref{eqn9}) hold. We have
		\begin{eqnarray}\label{eqn37}
			\left(\Sigma^{\ast}_n\right)^{-\frac{1}{2}}\left(
			\begin{array}{c}
				\sqrt{N_{A_2,n}}\left((\hat{m}_{A_1,n}-\hat{m}_{A_2,n})-(m_1-m_2)\right) \\
				\vdots \\
				\sqrt{N_{A_d,n}}\left((\hat{m}_{A_{d-1},n}-\hat{m}_{A_d,n})-(m_{d-1}-m_d) \right)
			\end{array}\right)
			\overset{L}{\to}{\cal N}({\bf 0},{\bf I}_{d-1}),
		\end{eqnarray}
		where $\Sigma^{\ast}_n$ is a positive definite matrix, the diagonal elements are
		\begin{eqnarray}\label{eqn38}
			\left(\Sigma^{\ast}_n\right)_{q,q}=\begin{cases}
				\frac{N_{A_{q+1},n}}{N_{A_q,n}}\Sigma_{qq}-2\sqrt{\frac{N_{A_{q+1},n}}{N_{A_q},n}}\Sigma_{q,q+1}+\Sigma_{q+1,q+1},\ &\text{$q\le t-1,$}\\
				\Sigma_{q+1,q+1},\ &\text{$q=t,$}\\
				\frac{N_{A_{q+1},n}}{N_{A_q,n}}\Sigma_{qq}+\Sigma_{q+1,q+1},\ &\text{$q>t$}.
			\end{cases}
		\end{eqnarray}
		the non-diagonal element is $\left(\Sigma^{\ast}_n\right)_{p,q}=0$ for $q\ge t$. And for $q\le t-1$,
		\begin{align*}
			\left(\Sigma^{\ast}_n\right)_{p,q}=
			\sqrt{\frac{N_{A_{p+1},n}N_{A_{q+1},n}}{N_{A_p,n}N_{A_q,n}}}\Sigma_{pq}-\sqrt{\frac{N_{A_{q+1},n}}{N_{A_q,n}}}\Sigma_{p+1,q}-\sqrt{\frac{N_{A_{p+1},n}}{N_{A_p,n}}}\Sigma_{p,q+1}+\Sigma_{p+1,q+1},
		\end{align*}
		where the matrix $\Sigma$ is define in Theorem \ref{thm5}. And, the convergence of \eqref{eqn37} still holds if $\Sigma^{\ast}_n$ is replaced by $\hat\Sigma^{\ast}_n$, that is, the random varibles  $N$, $Q$, $Z_p(p=1,\dots,t)$, $c_{pq}(1\le p<q\le t)$  and $\sigma^2_p(p=1,\dots,d)$ are replaced by their strongly consistency estimators $\hat{\mu}_{n}$, $\hat{q}_{N,n}$, $\hat{\nu}_{np}$, $\hat{q}_{A_pA_q,n}-\hat{m}_{A_p,n}\hat{m}_{A_q,n}$ and $\hat{q}_{A_p,n}-\hat{m}^2_{A_p,n}$, respectively.
	\end{corollary}

	\begin{remark}\label{re1}
			From Corollary \ref{co5} we have the following results. Define
			\begin{equation}\label{eqn613}
				{\bf V}_n=\left(
				\begin{array}{c}
					\sqrt{N_{A_2,n}}\left((\hat{m}_{A_1,n}-\hat{m}_{A_2,n})-(m_1-m_2)\right) \\
					\vdots \\
					\sqrt{N_{A_{k_0},n}}\left((\hat{m}_{A_{k_0-1},n}-\hat{m}_{A_{k_0},n})-(m_{k_0-1}-m_{k_0}) \right)
				\end{array}\right).
			\end{equation}
			Then, when $k_0\le t$, then
			\begin{equation}\label{eqn614}
				\left(\Sigma^{\ast}_{n,0}\right)^{-\frac{1}{2}}{\bf V}_n
				\overset{L}{\to}{\cal N}({\bf 0},{\bf I}_{k_0-1}),
			\end{equation}
			where for $q\in\{1,\dots,k_0-1\}$,
			\begin{eqnarray*}
				\left(\Sigma^{\ast}_{n,0}\right)_{q,q}=\frac{N_{A_{q+1},n}}{N_{A_q,n}}\Sigma_{qq}-2\sqrt{\frac{N_{A_{q+1},n}}{N_{A_q},n}}\Sigma_{q,q+1}+\Sigma_{q+1,q+1},
			\end{eqnarray*}
			and for $p,q\in\{1,\dots,k_0-1\}$, $p\neq q$,
			\begin{equation*}
				\left(\Sigma^{\ast}_{n,0}\right)_{p,q}=
				\sqrt{\frac{N_{A_{p+1},n}N_{A_{q+1},n}}{N_{A_p,n}N_{A_q,n}}}\Sigma_{pq}-\sqrt{\frac{N_{A_{q+1},n}}{N_{A_q,n}}}\Sigma_{p+1,q}-\sqrt{\frac{N_{A_{p+1},n}}{N_{A_p,n}}}\Sigma_{p,q+1}+\Sigma_{p+1,q+1}.
		\end{equation*}

			When $k_0> t$, then 
			\begin{equation}\label{eqn615}
				\left(\Sigma^{\ast}_{n,1}\right)^{-\frac{1}{2}}{\bf V}_n
				\overset{L}{\to}{\cal N}({\bf 0},{\bf I}_{k_0-1}),
			\end{equation}
			where for $q\in\{1,\dots,k_0-1\}$,
			\begin{eqnarray*}
				\left(\Sigma^{\ast}_{n,1}\right)_{q,q}=\begin{cases}
					\frac{N_{A_{q+1},n}}{N_{A_q,n}}\Sigma_{qq}-2\sqrt{\frac{N_{A_{q+1},n}}{N_{A_q},n}}\Sigma_{q,q+1}+\Sigma_{q+1,q+1},\ &\text{$q\le t-1,$}\\
					\Sigma_{q+1,q+1},\ &\text{$q=t,$}\\
					\frac{N_{A_{q+1},n}}{N_{A_q,n}}\Sigma_{qq}+\Sigma_{q+1,q+1},\ &\text{$t<q<k_0$}.
				\end{cases}
			\end{eqnarray*}
			the non-diagonal element is $\left(\Sigma^{\ast}_{n,1}\right)_{p,q}=-\sqrt{\frac{N_{A_{q+1},n}}{N_{A_q,n}}}\Sigma_{p+1,q}$ for $q\ge t$. And for $q\le t-1$,
			\begin{align*}
				\left(\Sigma^{\ast}_{n,1}\right)_{p,q}=
				\sqrt{\frac{N_{A_{p+1},n}N_{A_{q+1},n}}{N_{A_p,n}N_{A_q,n}}}\Sigma_{pq}-\sqrt{\frac{N_{A_{q+1},n}}{N_{A_q,n}}}\Sigma_{p+1,q}-\sqrt{\frac{N_{A_{p+1},n}}{N_{A_p,n}}}\Sigma_{p,q+1}+\Sigma_{p+1,q+1}.
		\end{align*}

			And, the convergence of \eqref{eqn614} and \eqref{eqn615} still hold if $\Sigma^\ast_{n,1}$ and $\Sigma^\ast_{n,2}$ are replaced by the estimators $\hat\Sigma^\ast_{n,1}$ and $\hat\Sigma^\ast_{n,2}$, respectively.
	\end{remark}

	\section{Hypothesis Testing}\label{sec5}
	We investigate a quintessential reinforcement learning subproblem: the multi-armed bandits. Here, a gambler faces $d$ slot machines, each with its distinct expected payoff. The gambler has $T$ opportunities to play and aims to maximize cumulative earnings. In each round, the player selects one machine, invests a coin, pulls the lever, and observes the resulting payoff.
	In an ideal scenario where the expected payoffs are known, the optimal strategy involves repeatedly selecting the machine with the highest expected return. However, each pull's outcome is inherently stochastic, reflecting the probabilistic nature of real-life decision-making where outcomes of choices are uncertain and only become clear post-action. Moreover, each round's decision provides information solely about the chosen machine, leaving the performance of others speculative. Thus, the gambler must balance between exploiting machines with historically higher returns and exploring lesser-used machines to potentially discover more lucrative options.
	
	The urn model proposed in this paper mirrors the mechanics of multi-armed bandits through a reinforcement mechanism. We associate each machine with a distinct colored ball within an urn, initially filled with $n_0$ balls of each color. For each color $k \in \{1,2,\dots,d\}$, the fixed expected payoff is denoted by $m_k$, while the actual payoff from pulling the corresponding machine is modeled by a random distribution $\mu_k$. Here, 'pulling' an arm equates to drawing a ball from the urn, and the payoff relates to the process of adding balls back into the urn.
	Initially, without prior knowledge, each arm is equally likely to be chosen. Suppose the first set of choices is represented by ${\bf X}_1 = (X_{11}, X_{12}, \dots, X_{1d})$, leading to a payoff of $(A_{11}\mathbb{I}(X_{11}\neq0), A_{12}\mathbb{I}(X_{12}\neq0), \dots, A_{1d}\mathbb{I}(X_{1d}\neq0))$. The cumulative payoff for each arm at this stage becomes ${\bf H}_1 = (H_{11}, H_{12}, \dots, H_{1d})$, with each $H_{1k} = n_0 + X_{1k}A_{1k}$. In subsequent rounds, selections are guided by ${\bf H}_1$, with the number of times each arm is chosen modeled by ${\bf X}_2 \sim \text{Multi-Hyper}(|{\bf H}_1|, N_2, {\bf H}_1)$, reflecting a strategy that favors arms with higher cumulative payoffs while still allowing for exploration of less chosen options.
	The decision process at each stage is driven by the immediate past results, with arm selection probabilities adjusted based on the accrued payoffs, promoting arms that have yielded higher returns. The use of a multidimensional hypergeometric distribution ensures a systematic exploration across all machines, including those less frequently chosen. Employing the asymptotic results derived in Sections \ref{sec3} and \ref{sec4}, we proceed to formulate and test hypotheses concerning the expected payoffs of the $d$ arms.

	In this section, we explore a general hypothesis testing scenario for the expected payoffs across multiple arms of a multi-armed bandits. Specifically, for a fixed $k_0\in\{1,2,\dots,d\}$, we aim to test the following hypotheses:
	
	\begin{equation*}
		H_0:m_1=m_2=\dots=m_{k_0}
	\end{equation*}
	versus
	\begin{equation*}\label{test1}
		H_1:\mbox{exists}\ k\in\{1,\dots,k_0-1\},\ \mbox{ such\ that}\ m_1=\dots=m_k>m_{k+1}\ge\dots\ge m_{k_0}. 
	\end{equation*}
	
	\subsection{Under the null hypothesis}
	We define 
	\begin{equation*}
		\tilde{{\bf V}}_n=\left(\sqrt{N_{A_2,n}}\left(\hat{m}_{A_1,n}-\hat{m}_{A_2,n}\right),\dots,\sqrt{N_{A_{k_0},n}}\left(\hat{m}_{A_{k_0-1},n}-\hat{m}_{A_{k_0},n} \right)\right)^{\top}.
	\end{equation*}
	We choose the test statistic
	\begin{equation*}
		\Theta_n=\tilde{{\bf V}}^{\top}_n(\hat{\Sigma}^\ast_{n,0})^{-1}\tilde{{\bf V}}_n, 
	\end{equation*}
	where $\hat{\Sigma}^\ast_{n,0}$ is defined in Remark \ref{re1}. The critical region corresponding to the significance level $\alpha$ is
	\begin{equation*}
		W^{\ast}=\{\Theta_n>\chi^2_{1-\alpha}(k_0-1)\},
	\end{equation*}
	where $\chi^2_{1-\alpha}(k_0-1)$ is the $(1-\alpha)$ quantile of the chi-squared distribution with $k_0-1$ degrees of freedom. Thus, it is straightforward from  Remark \ref{re1} and the continuous mapping theorem, under $H_0$, we have that 
	\begin{equation*}
		\mathbb P(W^{\ast})\to\alpha.
	\end{equation*}

	\subsection{Under the alternative hypothesis}
	Next, we consider the distribution of the test statistic $\Theta_n$ under the alternative hypothesis
	$$H_1: m_1=\dots=m_k>m_{k+1}\geq \dots\ge m_{k_0},$$
	for some $\ k\in\{1,\dots,k_0-1\}$.
	We define 
	\begin{eqnarray*}
		{\bf \Delta}_n&=&\left(\sqrt{N_{A_2,n}}\left(m_1-m_2\right),\dots,\sqrt{N_{A_{k_0},n}}\left(m_{k_0-1}-m_{k_0}\right) \right)^{\top}\\
		&=&(0,\dots,0,\sqrt{N_{A_{k+1},n}}\left(m_{k}-m_{k+1}\right),\dots,\sqrt{N_{A_{k_0},n}}\left(m_{k_0-1}-m_{k_0}\right))^{\top}.
	\end{eqnarray*}
	Recall the definition of ${\bf V}_n$ and $\hat{\Sigma}^\ast_{n,1}$ in Remark \ref{re1}, then ${\bf V}_n=\tilde{{\bf V}}_n+{\bf \Delta}_n$,
	
	\begin{align}\label{eqn45}
		&\Theta_n=\tilde{{\bf V}}^{\top}_n(\hat{\Sigma}^\ast_{n,0})^{-1}\tilde{{\bf V}}_n\non\\
		=&({\bf V}_n-{\bf \Delta}_n)^{\top}[(\hat{\Sigma}^\ast_{n,1})^{-1}+(\hat{\Sigma}^\ast_{n,0})^{-1}-(\hat{\Sigma}^\ast_{n,1})^{-1}]({\bf V}_n-{\bf \Delta}_n)\non\\
		=&{\bf V}_n^{\top}(\hat{\Sigma}^\ast_{n,1})^{-1}{\bf V}_n-2{\bf \Delta}_n^{\top}(\hat{\Sigma}^\ast_{n,0})^{-1}{\bf V}_n+{\bf \Delta}_n^{\top}(\hat{\Sigma}^\ast_{n,0})^{-1}{\bf \Delta}_n+
		{\bf V}_n^{\top}[(\hat{\Sigma}^\ast_{n,0})^{-1}-(\hat{\Sigma}^\ast_{n,1})^{-1}]{\bf V}_n
	\end{align}
	Under alternative hypothesis $H_1$, the first term of (\ref{eqn45}) equals to
	\begin{equation}
		{\bf V}_n^{\top}(\hat{\Sigma}^\ast_{n,1})^{-1}{\bf V}_n\overset{L}{\to}\chi^2(k_0-1).
	\end{equation}
	The second term of (\ref{eqn45}) equals to
	\begin{eqnarray*}
		-2{\bf \Delta}_n^{\top}(\hat{\Sigma}^{\ast}_{n,0})^{-1}{\bf V}_n=-2{\bf \Delta}_n^{\top}(\hat{\Sigma}^{\ast}_{n,0})^{-1/2}(\hat{\Sigma}^{\ast}_{n,0})^{-1/2}{\bf V}_n.
	\end{eqnarray*}
	
	For ease of description later, we denote the estimator of $\Sigma$ at $t=k_0$ as $\hat\Sigma$ and the estimator of $\Sigma$ at $t=k$ as $\breve\Sigma$. Note that under $H_1$, 
	\begin{eqnarray*}
		\hat{\Sigma}^{\ast}_{n,0}&\approx&\begin{pmatrix}
			\left(\hat{\Sigma}^{\ast(1)}_{n,0}\right)_{(k-1)\times (k-1)} & {\bf 0}\\
			{\bf 0} & {\rm diag}\left\{	\frac{N_{A_{k+1},n}}{N_{A_k,n}}\hat\Sigma_{kk}+\hat{\sigma}^2_{k+1},\dots,\frac{N_{A_{k_0},n}}{N_{A_{k_0-1},n}}\hat\sigma^2_{k_0-1}+\hat{\sigma}^2_{k_0}\right\} 
		\end{pmatrix},
	\end{eqnarray*}
	we have
	\begin{eqnarray*}
		{\bf \Delta}_n^{\top}(\hat{\Sigma}^{\ast}_{n,0})^{-1/2}\approx\left(0,\dots,0,\frac{\sqrt{N_{A_{k+1},n}}(m_k-m_{k+1})}{\sqrt{\frac{N_{A_{k+1},n}}{N_{A_k,n}}\hat\Sigma_{kk}+\hat{\sigma}^2_{k+1}}},\dots,\frac{ \sqrt{N_{A_{k_0},n}}(m_{k_0-1}-m_{k_0})}{\sqrt{\frac{N_{A_{k_0},n}}{N_{A_{k_0-1},n}}\hat\sigma^2_{k_0-1}+\hat{\sigma}^2_{k_0}}}\right),
	\end{eqnarray*}
	where for all $j\in\{k+1,\dots,k_0\}$, $$N_{A_j,n}/n^{m_j/m_1}\overset{a.s.}{\to}m_1/m_jN\tilde{Z}_j.$$
	And, we have $	(\hat{\Sigma}^{\ast}_{n,0})^{-1/2}{\bf V}_n=O_p(1)$.
	Thus, given random varibles $N,\tilde{Z}_{k+1},\dots,\tilde{Z}_{k_0}$, 
	\begin{equation}\label{eqn47}
		{\bf \Delta}_n^{\top}(\hat{\Sigma}^{\ast}_{n,0})^{-1}{\bf V}_n= O_p(n^{m_{k+1}/(2m_1)}).
	\end{equation}
	For the third term of (\ref{eqn45}), given random varibles $N,\tilde{Z}_{k+1},\dots,\tilde{Z}_{k_0}$, 
	\begin{equation}
		{\bf \Delta}_n^{\top}(\hat{\Sigma}^\ast_{n,0})^{-1}{\bf \Delta}_n\approx\frac{N_{A_{k+1},n}(m_k-m_{k+1})^2}{\frac{N_{A_{k+1},n}}{N_{A_k,n}}\hat\Sigma_{kk}+\hat{\sigma}^2_{k+1}}+\dots+\frac{N_{A_d,n}(m_{{k_0}-1}-m_{k_0})^2}{\frac{N_{A_{k_0},n}}{N_{A_{k_0-1},n}}\hat\sigma^2_{k_0-1}+\hat{\sigma}^2_{k_0}}=O_{a.s.}(n^{m_{k+1}/m_1}).
	\end{equation}
	In order to calculate the last term of (\ref{eqn45}), we first consider the term $\hat{\Sigma}^\ast_{n,0}-\hat{\Sigma}^\ast_{n,1}$. 
	Recall the definition of $\hat{\Sigma}^\ast_{n,0}$ and $\hat{\Sigma}^\ast_{n,1}$, we have that the diagonal element of $\hat{\Sigma}^\ast_{n,0}-\hat{\Sigma}^\ast_{n,1}$ is 0 for $q\le k-1$. For $k\le q<k_0$,
	\begin{eqnarray*}
		&&(\hat{\Sigma}^\ast_{n,0}-\hat{\Sigma}^\ast_{n,1})_{q,q}\\
		&\approx&\begin{cases}
			\frac{N_{A_{k+1},n}}{N_{A_k,n}}\hat{\Sigma}_{kk}-2\sqrt{\frac{N_{A_{k+1},n}}{N_{A_k},n}}\hat{\Sigma}_{k,k+1}+\hat{\Sigma}_{k+1,k+1}-\frac{N_{A_{k+1},n}}{N_{A_k,n}}\breve{\Sigma}_{kk}-\breve{\Sigma}_{k+1,k+1}, \ &\text{$q=k,$}\\
			\frac{N_{A_{q+1},n}}{N_{A_q,n}}\hat{\Sigma}_{qq}-2\sqrt{\frac{N_{A_{q+1},n}}{N_{A_q},n}}\hat{\Sigma}_{q,q+1}+\hat{\Sigma}_{q+1,q+1}-\frac{N_{A_{q+1},n}}{N_{A_q,n}}\breve{\Sigma}_{qq}-\breve{\Sigma}_{q+1,q+1},\ &\text{$k<q<k_0$}.
		\end{cases}
	\end{eqnarray*}
	Note that under $H_1$, $N_{A_{k+1},n}/N_{A_k,n}\overset{a.s.}{\to}0$, and
	for all $q\ge k$, we have $\hat\Sigma_{qq}=O_{a.s.}(1)$, $\hat\Sigma_{q,q+1}=O_{a.s.}(1)$,  and
	\begin{equation*}
		\hat\Sigma_{kk}=\breve\Sigma_{kk},\ \ 	\hat\Sigma_{q+1,q+1}-\breve\Sigma_{q+1,q+1}=\hat\sigma^2_{q+1}[\hat Z_{q+1}(\hat{Q}/\hat{N}-1)+1]-\hat\sigma^2_{q+1}\overset{a.s.}{\to}0.
	\end{equation*}
	Thus, for all $q\in\{1,2,\dots,d-1\}$, 
	\begin{equation}\label{eqnew1}
		(\hat{\Sigma}^\ast_{n,0}-\hat{\Sigma}^\ast_{n,1})_{q,q}\overset{a.s.}{\to}0.
	\end{equation}
	Now we consider the non-diagonal elements of $\hat{\Sigma}^\ast_{n,0}-\hat{\Sigma}^\ast_{n,1}$. For all $1\le p<q\le k_0-1$, we have $(\hat{\Sigma}^\ast_{n,0}-\hat{\Sigma}^\ast_{n,1})_{p,q}=0$ for $q\le k-1$. And for $q\ge k$,
	\begin{eqnarray*}
		(\hat{\Sigma}^\ast_{n,0}-\hat{\Sigma}^\ast_{n,1})_{p,q}&=&
		\sqrt{\frac{N_{A_{p+1},n}N_{A_{q+1},n}}{N_{A_p,n}N_{A_q,n}}}\hat{\Sigma}_{pq}-\sqrt{\frac{N_{A_{q+1},n}}{N_{A_q,n}}}\hat{\Sigma}_{p+1,q}-\sqrt{\frac{N_{A_{p+1},n}}{N_{A_p,n}}}\hat{\Sigma}_{p,q+1}\\
		&&+\hat{\Sigma}_{p+1,q+1}+\sqrt{\frac{N_{A_{q+1},n}}{N_{A_q,n}}}\breve{\Sigma}_{p+1,q},
	\end{eqnarray*}
	Under $H_1$, for all $q> k$, we have $$\hat\Sigma_{pq}=o_{a.s.}(1),\  \frac{N_{A_{q+1},n}}{N_{A_q,n}}=O_{a.s.}(1),\ \hat{\Sigma}_{p+1,q}-\breve{\Sigma}_{p+1,q}\overset{a.s.}{\to}0,$$
	and 
	$$\hat\Sigma_{p,q+1}=\hat c_{p,q+1}\sqrt{\hat Z_p\hat Z_{q+1}}(\hat Q/\hat{N}-1)\overset{a.s.}{\to}0,\ \ \hat\Sigma_{p+1,q+1}=\hat c_{p,q+1}\sqrt{\hat Z_p\hat Z_{q+1}}(\hat Q/\hat N-1)\overset{a.s.}{\to}0.$$
	Then, under $H_1$, for all $1\le p<q\le k_0-1$,
	\begin{equation} \label{eqnew2}
		(\hat{\Sigma}^\ast_{n,0}-\hat{\Sigma}^\ast_{n,1})_{p,q}\overset{a.s.}{\to}0.
	\end{equation}
	By (\ref{eqnew1}) and (\ref{eqnew2}), we conclude 
	$$\hat{\Sigma}^\ast_{n,0}-\hat{\Sigma}^\ast_{n,1}\overset{a.s.}{\to}\bf 0.$$
	Now we calculate the last term of (\ref{eqn45}). By Corollary \ref{co5}, we have that ${\bf V}_n$ converges in distribution. Thus, by Slutsky Theorem, we obtain
	\begin{equation}\label{eqn49}
		{\bf V}_n^{\top}[(\hat{\Sigma}^\ast_{n,0})^{-1}-(\hat{\Sigma}^\ast_{n,1})^{-1}]{\bf V}_n={\bf V}_n^{\top}(\hat{\Sigma}^\ast_{n,0})^{-1}\left(\hat{\Sigma}^\ast_{n,1}-\hat{\Sigma}^\ast_{n,0}\right)(\hat{\Sigma}^\ast_{n,1})^{-1}{\bf V}_n        \overset{p}{\to}0.
	\end{equation}
	
	By (\ref{eqn45})--(\ref{eqn49}), under $H_1$,
	\begin{equation*}
		\Theta_n\overset{d}{=}{\bf V}_n^{\top}(\hat{\Sigma}^\ast_{n,1})^{-1}{\bf V}_n+{\bf \Delta}_n^{\top}(\hat{\Sigma}^\ast_{n,0})^{-1}{\bf \Delta}_n,
	\end{equation*}
	where the first term is asymptotically chi-square distribution and the second term tends almost surely to infinity with order $n^{m_{k+1}/m_1}$ when $n$ tends to $\infty$. Thus, given random variables $N$, $\tilde Z_{k+1},\dots,\tilde Z_{k_0}$, we conclude that the power of this test tends to 1, effectively distinguishing between the homogeneity and hierarchical inequality of payoffs among the arms.

	\section{Simulation Results}\label{sec6}
	This section presents simulation results to validate the theoretical findings discussed in Section \ref{sec5}, focusing on the normalized urn composition and the power of the hypothesis tests. We specifically address the case where $d=3$, considering various scenarios for the distribution of replacement numbers.

	\subsection{Simulation results for the normalized urn composition}
	We consider the case $m1=m2>m3$. In Section \ref{sec3}, we have proved the existence of $Z_1$ and $Z_2$ without their specific distribution. Therefore, next, we give some simulation results under specific parameter settings as shown in Fig. \ref{figure1} and \ref{figure2}. In Fig. \ref{figure1}, we assume that ${\bf Y}_0=(6,6,6)$. And in Fig. \ref{figure2}, we assume that the initial composition is ${\bf Y}_0=(6,3,6)$. In addition, in both Fig. \ref{figure1} and \ref{figure2}, we take the following settings.  
	Taken $\{N_n,n=1,2,\dots\}$ independently and uniformly distributed on $\{3,6\}$, and
	$$\mathbb P(N_n=3)=1/3,\ \mathbb P(N_n=6)=2/3.$$
	
	\begin{enumerate}[label=Case \alph*.]
		\item
		Taken $\{A_{n1},n=1,2,\dots\}$ and $\{A_{n2},n=1,2,\dots\}$ independently and uniformly distributed on $\{3,4,5\}$ with equal probability $1/3$, respectively. Moreover, taken $A_{n3}=1$ for all $n$.

		\item Taken $\{A_{n1},n=1,2,\dots\}$ independently and uniformly distributed on $\{2,3,4,5,6\}$ with equal probability $1/5$. Taken $\{A_{n2},n=1,2,\dots\}$ independently and uniformly distributed on $\{3,4,5\}$ with equal probability $1/3$.
		Moreover, taken $A_{n3}=1$ for all $n$.

		\item For $k=1,2,3$, let $Y_k$ be the $k$th component of ${\rm Multinomial}(5;(0.4,0.4,0.2))$.
		Taken $\{A_{nk}, n=1,2,\dots\}$ independently and identically distributed, and 
		$$A_{n1}\overset{d}{=}2+Y_1,\ A_{n2}\overset{d}{=}2+Y_2,\ A_{n3}=1.$$

		\item For $k=1,2,3$, let $Y_k$ be the $k$th component of ${\rm Multinomial}(5;(0.3,0.4,0.3))$.
		Taken $\{A_{nk}, n=1,2,\dots\}$ independently and identically distributed, and 
		$$A_{n1}\overset{d}{=}1+2Y_1,\ A_{n2}\overset{d}{=}2+Y_2,\ A_{n3}=1.$$
	\end{enumerate}

	\begin{figure}
		\centering
		\includegraphics[width=0.8\textwidth]{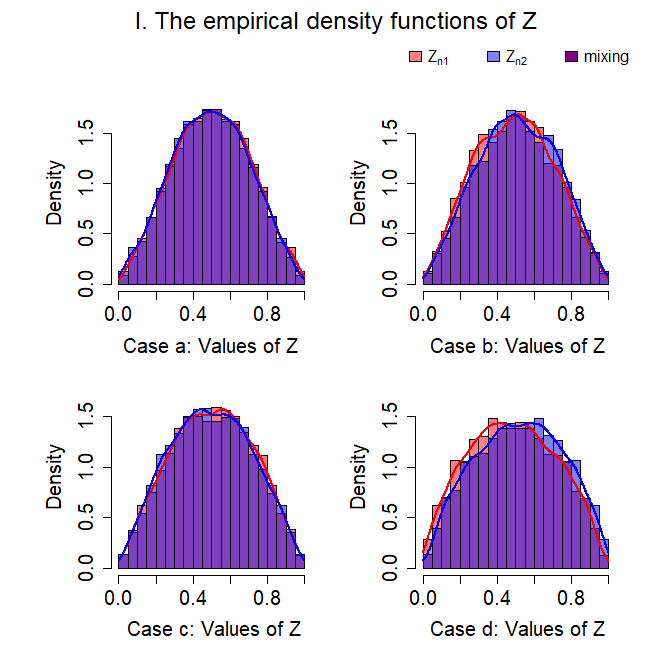}
		\captionsetup{skip=-6pt}
		\caption{Histograms and density plots of the normalized urn compositions $Z_{n1}$ and $Z_{n2}$. We set the initial urn composition ${\bf Y}_0=(6,6,6)$. We select 5000 samples, each with 10,000 stages of urn updates, i.e., $n=10,000$.}
		\label{figure1}
	\end{figure}

	\begin{figure}
		\centering
		\includegraphics[width=0.8\textwidth]{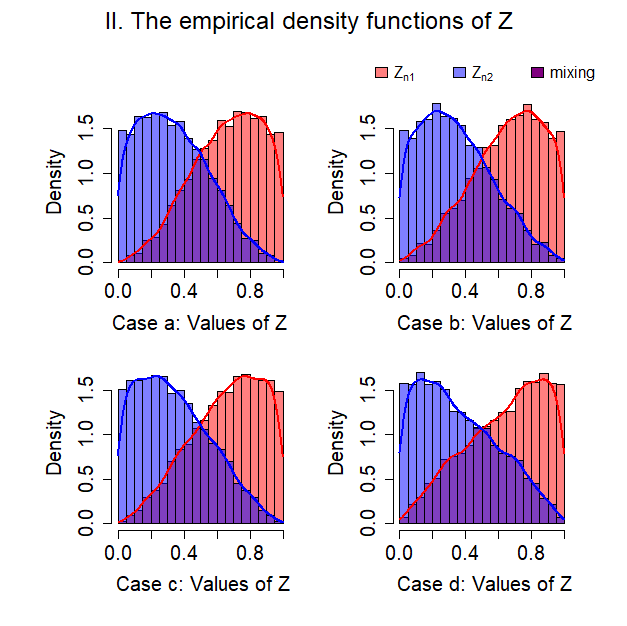}
		\captionsetup{skip=-6pt}
		\caption{Histograms and density plots of the normalized urn compositions $Z_{n1}$ and $Z_{n2}$. We set the initial urn composition ${\bf Y}_0=(6,3,6)$. We select 5000 samples, each with 10,000 stages of urn updates, i.e., $n=10,000$.}
		\label{figure2}
	\end{figure}

	In Figures \ref{figure1} and \ref{figure2}, the only different setting is whether the initial values of the first and second type of balls are the same. From the results in Fig. \ref{figure1} and \ref{figure2}, it can be seen that the initial value has a significant effect on the limiting distribution. In Case a and Case b, the components of the replacement matrix ${\bf A}_n$ are uncorrelated. Conversely, in Case c and Case d, the components exhibit correlation. Comparing Cases a, b and Cases c, d of Figures \ref{figure1} and \ref{figure2}, it can be seen that the correlation of the replacement matrices may affect the distribution of the limits. Additionally, in Case a and Case c of Figures \ref{figure1} and \ref{figure2}, the variables $A_{n1}$ and $A_{n2}$ are assumed to be identically distributed. Therefore, if the initial values of the first two balls are the same, then $Z_1$ and $Z_2$ should be identically distributed. This hypothesis is supported by the results shown in Figure \ref{figure1} for Cases a and c. In contrast, in Case b and Case d, while only the first-order moments of $A_{n1}$ and $A_{n2}$ are identical. Comparing Case a and Case b or Case c and Case d of Figures \ref{figure1} and \ref{figure2}, respectively, it is possible to see that the distribution of the limit for a given initial value depends entirely on the first-order moments of the replacement matrix.
	
	\subsection{Simulation results for power}
	In this section, we give the simulation results for
	the power of the test 
	\begin{equation}\label{eqn612}
		H_0:m_1=m_2=m_3\ \  \mbox{versus}\ \  H_1:m_1=m_2>m_3.
	\end{equation}
	under the following parameter settings. Let ${\bf H}_0=(9,9,9)$. 	Let $\{N_n,n=1,2,\dots\}$ independently and uniformly distributed on $\{6,8\}$, and
	$$\mathbb P(N_n=6)=1/4,\ \mathbb P(N_n=8)=3/4.$$ 
	Define $Y\sim {\rm Poisson}(6)$.
	For $k=1,2,3$, taken $\{A_{nk},n=1,2,\dots\}$ independently and uniformly distributed, and 
	$$A_{n1}\overset{d}{=} Y+1,\ A_{n2}\overset{d}{=} Y+1,\ A_{n3}\overset{d}{=} Y+1-0.2e,$$
	where we gradually choose $e$ from 1 to 10.

	We can see from Figure \ref{figure4} that as the value of $m_2-m_3$ increases, the power of the test \eqref{eqn612} approaches 1. Remarkably, even when $m_2-m_3$ is minimal, with a sample size of 500 and 1,000 values at each sample point (equivalent to 1,000 updates of the urn), the power is already nearly 1. Therefore, the test remains effective even when the difference in means is small.

	\begin{figure}
		\centering
		\includegraphics[width=1.0\textwidth]{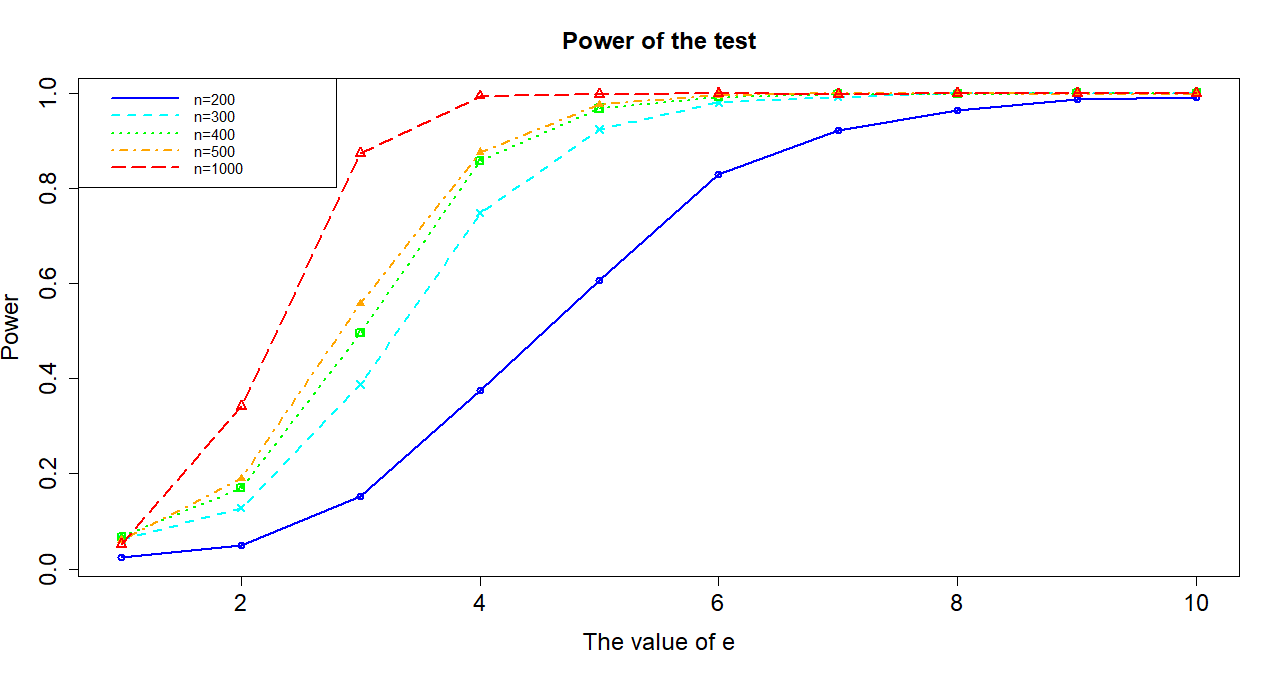}
		\captionsetup{skip=-3pt}
		\caption{Power for test \eqref{eqn612}. We select 500 samples, each with 1000 stages of urn updates. In addition, we define $e=5(m_2-m_3)$.}
		\label{figure4}
	\end{figure}

	\section*{Acknowledgments}
	Zhidong Bai was partially supported by NSFC Grants No.12171198, No.12271536, and Team Project of Jilin Provincial Department of Science and Technology No.20210101147JC. Jiang Hu was partially supported by NSFC Grants No. 12171078, No. 12292980, No. 12292982, National Key R $\&$ D Program of China No. 2020YFA0714102, and Fundamental Research Funds for the Central Universities No. 2412023YQ003.

	\vspace{0.6cm}
	\begin{appendix}
		\renewcommand\thelemma{A\arabic{lemma}}
		\setcounter{equation}{0}
		\renewcommand\theequation{A\arabic{equation}}

		\section{Proofs of main results}\label{secA0}
		In this section, we give proofs of the main results.
		\begin{proof}[Proof of Lemma \ref{le1}]
			For all $k\in\{1,2,\dots,d\}$,
			$H_{nk}=H_0+\sum_{i=1}^nA_{ik}X_{ik}\ge \sum_{i=1}^n\mathbb I(X_{ik}\ge1).$ It's trivial that $$S_n=S_0+\sum_{i=1}^n\sum_{k=1}^dA_{ik}X_{ik}\le S_0+C_1\sum_{i=1}^n\sum_{k=1}^dA_{ik}=O_{a.s.}(n)$$
			and 
			$$\mathbb P(X_{ik}=0|{\cal F}_{i-1})=\frac{{{S_{i-1}-H_{i-1,k}} \choose  N_i} }{ {S_{i-1} \choose N_i}}\le \frac{S_{i-1}-H_{i-1,k}}{S_{i-1}}\le \frac{S_{i-1}-H_{0k}}{S_{i-1}}.$$ 
			Thus, we obtain
			\begin{eqnarray*}
				&&\sum_{i=1}^n\mathbb E(\mathbb I(X_{ik}\ge1)|{\cal F}_{i-1})=\sum_{i=1}^n \mathbb P(X_{ik}\ge1|{\cal F}_{i-1})\\
				&=&\sum_{i=1}^n[1-\mathbb P(X_{ik}=0|{\cal F}_{i-1})]\ge\sum_{i=1}^n\frac{H_{0k}}{S_{i-1}}\to\infty,\ \ a.s..
			\end{eqnarray*}
			By Lemma \ref{LeB6}, we conclude $\sum_{i=1}^n\mathbb I(X_{ik}\ge1)\overset{a.s.}{\to}\infty$, then $H_{nk}(k\in\{1,2,\dots,d\})$ converges to $\infty$ almost surely.
		\end{proof}
		\begin{proof}[Proof of Theorem \ref{th1}]
			For all $k\in\{1,2,\dots,t\}$, we have
			\begin{eqnarray}\label{eqqx1}
				&&Z_{n+1,k}-Z_{nk}=\frac{H_{n+1,k}}{S_{n+1}}-Z_{nk}\\
				&=&\frac{H_{nk}+A_{n+1,k}X_{n+1,k}-Z_{nk}\left(S_n+\sum_{j=1}^dA_{n+1,j}X_{n+1,j}\right)}{S_{n+1}}\non\\
				&=&\frac{(1-Z_{nk})A_{n+1,k}X_{n+1,k}-Z_{nk}\sum_{j\neq k}A_{n+1,j}X_{n+1,j}}{S_{n+1}}\non\\
				&=&\frac{(1-Z_{nk})A_{n+1,k}X_{n+1,k}-Z_{nk}\sum_{j\neq k}A_{n+1,j}X_{n+1,j}}{S_n}\non\\
				&&\quad-\Big(\frac{1}{S_n}-\frac{1}{S_{n+1}}\Big)\Big[(1-Z_{nk})A_{n+1,k}X_{n+1,k}-Z_{nk}\sum_{j\neq k}A_{n+1,j}X_{n+1,j}\Big],
			\end{eqnarray}
			where $X_{n+1,d}=N_{n+1}-\sum\limits_{j=1}^{d-1}X_{n+1,j}$. We define
			\begin{eqnarray*}
				e_{n+1,k}&=&\Big(\frac{1}{S_n}-\frac{1}{S_{n+1}}\Big)\Big[(1-Z_{nk})A_{n+1,k}X_{n+1,k}-Z_{nk}\sum_{j\neq k}A_{n+1,j}X_{n+1,j}\Big],\\
				E_{nk}&=&\sum_{i=1}^ne_{ik},\ \ L_{nk}=Z_{nk}+E_{nk}.
			\end{eqnarray*}
			Then, the random varible $L_{nk}$ is ${\cal G}_n$-measurable. Since
			\begin{equation}\label{eqqx5}
				S_n=S_0+\sum\limits_{i=1}^n\sum\limits_{k=1}^dA_{ik}X_{ik}\ge S_0+\sum\limits_{i=1}^n\sum\limits_{k=1}^dX_{ik}\ge S_0+n\ge n+1,
			\end{equation}
			we obtain
			\begin{eqnarray}\label{eqqx3}
				&&E_{nk}=\sum_{i=1}^ne_{ik}\le \sum_{i=1}^n|e_{ik}|\non\\
				&\le&\sum_{i=1}^n\frac{(1-Z_{i-1k})A_{ik}X_{ik}\sum_{j=1}^dA_{ij}X_{ij}+Z_{i-1,k}\sum_{s\neq k}\sum_{j=1}^dA_{is}A_{ij}X_{is}X_{ij}}{S_{i-1}S_i}\non\\
				&\le&\sum_{i=1}^n\frac{C^2_1\left(\sum_{j=1}^dA_{ik}A_{ij}+\sum_{s\neq k}\sum_{j=1}^dA_{is}A_{ij}\right)}{i^2}.
			\end{eqnarray}
			Since the second-order moments of the replacement matrices are bounded, so $E_{nk}$ converges almost surely, which also implies that $\{L_{nk},n\ge1\}$ is ${\cal L}_1$ bounded. In addition, we have
			\begin{eqnarray}\label{eqqx2}
				&&\mathbb E(L_{n+1,k}-L_{nk}|{\cal G}_n)=\mathbb E(Z_{n+1,k}-Z_{nk}+e_{n+1,k}|{\cal G}_n)\non\\
				&=&\mathbb E\Big[\frac{(1-Z_{nk})A_{n+1,k}X_{n+1,k}-Z_{nk}\sum_{j\neq k}A_{n+1,j}X_{n+1,j}}{S_n}|{\cal G}_n\Big]\non\\
				&=&\frac{N_{n+1}}{S_n}\Big[(1-Z_{nk})Z_{nk}m_{n+1,k}-\sum_{j\neq k}Z_{nk}Z_{nj}m_{n+1,j}\Big]\non\\
				&=&\frac{N_{n+1}}{S_n}\Big[\sum_{j=t+1}^dZ_{nk}Z_{nj}(m_{n+1,k}-m_{n+1,j})\Big],
			\end{eqnarray}
			where the last equation is because that by Assumption \ref{as2}, $m_1=\dots=m_t>m_{t+1}\ge \dots\ge m_d$ holds, and
			$$(1-Z_{nk})Z_{nk}m_{n+1,k}-\sum_{j\neq k}Z_{nk}Z_{nj}m_{n+1,k}=0.$$
			Thus, the random sequence $\{L_{nk},n\ge1\}$ is eventually a ${\cal L}_1$ bounded submartingale. 
			Then it converges almost surely to a random varible. By the definition of $L_{nk}$ and the almost sure convergence of $E_n$, we obtain the almost sure convergence of $Z_{nk}$.
		\end{proof}

		\begin{proof}[Proof of Theorem \ref{th2}]
			For a fixed  $\theta\in\left(\frac{m_{t+1}}{m_1},1\right)$, we define  $G_n=\frac{\sum_{k=t+1}^{d}H_{nk}}{H^{\theta}_{n1}}$, then
			\begin{eqnarray*}
				\frac{G_{n+1}}{G_n}&=&\frac{\sum_{k=t+1}^{d}H_{n+1,k}}{H^{\theta}_{n+1,1}} \frac{H^{\theta}_{n1}}{\sum_{k=t+1}^{d}H_{nk}}\\
				&=&\left(\frac{H_{n1}}{H_{n+1,1}}\right)^{\theta}\frac{\sum_{k=t+1}^{d}H_{nk}+\sum_{k=t+1}^{d}A_{n+1,k}X_{n+1,k}}{\sum_{k=t+1}^{d}H_{nk}}.	
			\end{eqnarray*}
			Thus, we have 
			\begin{eqnarray*}
				&&\mathbb E\left(\frac{G_{n+1}}{G_n}-1|{\cal G}_n\right)\\
				&\le&\mathbb E\left[\left(\frac{H_{n1}}{H_{n+1,1}}\right)^{\theta}-1|{\cal G}_n\right]+\mathbb E\left(\frac{\sum_{k=t+1}^{d}A_{n+1,k}X_{n+1,k}}{\sum_{k=t+1}^{d}H_{nk}}|{\cal G}_n\right)\\
				&\le&-\theta\mathbb E\left(\frac{A_{n+1,1}X_{n+1,1}}{H_{n+1,1}}|{\cal G}_n\right)+\frac{N_{n+1}}{S_n} \frac{\sum_{k=t+1}^{d}m_{n+1,k}H_{nk}}{\sum_{k=t+1}^{d}H_{nk}}\\
				&=&-\theta\mathbb E\left(\frac{A_{n+1,1}X_{n+1,1}}{H_{n,1}}|{\cal G}_n\right)+\frac{N_{n+1}}{S_n} \frac{\sum_{k=t+1}^{d}m_{n+1,k}H_{nk}}{\sum_{k=t+1}^{d}H_{nk}}\\
				&&\quad+\theta\mathbb E\left[\left(\frac{1}{H_{n1}}-\frac{1}{H_{n+1,1}}\right)A_{n+1,1}X_{n+1,1}|{\cal G}_n\right]\\
				&=&\frac{N_{n+1}}{S_n}\left(\frac{\sum_{k=t+1}^{d}m_{n+1,k}H_{nk}}{\sum_{k=t+1}^{d}H_{nk}}-\theta m_{n+1,1}\right)+\theta\mathbb E\left[\frac{A^2_{n+1,1}X^2_{n+1,1}}{H_{n1}H_{n+1,1}}|{\cal G}_n\right]\\
				&\le&\frac{N_{n+1}}{S_n}\left(\frac{\sum_{k=t+1}^{d}m_{n+1,k}H_{nk}}{\sum_{k=t+1}^{d}H_{nk}}-\theta m_{n+1,1}\right)+\frac{\theta}{H^2_{n1}}\mathbb E(A^2_{n+1,1})\mathbb E(X^2_{n+1,1}|{\cal G}_n)\\
				&\le&\frac{N_{n+1}}{S_n}\left(\frac{\sum_{k=t+1}^{d}m_{n+1,k}H_{nk}}{\sum_{k=t+1}^{d}H_{nk}}-\theta m_{n+1,1}+\frac{2\theta C_1(C_1+1)C_2}{H_{n1}}\right),
			\end{eqnarray*}
			where the first inequality is due to that $(1-x)^\theta\le 1-\theta x$ holds for all $0<x<1$, and the third inequality is because that for all $j\in\{1,2,\dots,d\}$,
			\begin{equation*}
				\mathbb E(X^2_{n+1,1}|{\cal G}_n)\le C_1 \mathbb E(X_{n+1,1}|{\cal G}_n)=C_1 N_{n+1}Z_{n1}=C_1N_{n+1}H_{n1}/S_n.
			\end{equation*}
			By Assumption \ref{as2}, we have $m_{n+1,k}\to m_k$ holds for all $k\in\{1,2,\dots,d\}$ and $m_1=m_2=\dots=m_t>m_{t+1}\ge\dots\ge m_d$. Since $H_{n1}$ converges to $\infty$ almost surely, for large enough $n$, we obtain that $\mathbb E\left(G_{n+1}/G_{n}-1\right)$ is less than or equal to $0$. Thus, for large enough $n$,
			\begin{equation*}
				\mathbb E(G_{n+1}-G_n|{\cal G}_n)=G_n\mathbb E(G_{n+1}/G_n-1)\le0,\ \ \mbox{a.s.}.
			\end{equation*}
			Then, $G_n$ is eventually a supermartingale, then it converges almost surely to a finite random variable $Q$.  By (\ref{eqqx5}), we have $S_n$ converges to $\infty$, then
			\begin{eqnarray*}
				1-\sum\limits_{k=1}^{t}Z_{nk}=\frac{\sum_{k=t+1}^{d}H_{nk}}{S_n}=\frac{\sum_{k=t+1}^{d}H_{nk}}{H^{\theta}_{n1}}\frac{H^{\theta}_{n1}}{S_n}\le G_nS^{-1+\theta}_n\to0,\ \ a.s.,
			\end{eqnarray*}
			thus $\sum_{k=1}^{t}Z_{nk}$ converges almost surely to 1. Theorem \ref{th2} is proved.
		\end{proof}

		\begin{proof}[Proof of Corollary \ref{co1}]
			For all $k\in\{1,2,\dots,d\}$ and $i=1,2,\dots,$
			\begin{eqnarray*}
				\mathbb E(A_{ik}X_{ik}|{\cal F}_{i-1})&=& \mathbb E[\mathbb E(A_{ik}X_{ik}|{\cal G}_{i-1})|{\cal F}_{i-1}]=m_{ik}Z_{i-1,k}\mathbb E(N_i|{\cal F}_{i-1})\\
				&\overset{a.s.}{\to}&\begin{cases}
					Nm_kZ_k, &\text{if $k\in\{1,\dots,t\}$},\\
					0,&\text{if $k\in\{t+1,\dots,d\}$}.
				\end{cases}
			\end{eqnarray*}
			Note that
			$$\frac{H_{nk}}{n}=\frac{1}{n}\sum_{i=1}^{n}A_{ik}X_{ik}+\frac{H_{0k}}{n},$$
			by Lmmma \ref{leB1}, where we choose $\alpha_j=1$ and $\beta_j=j$, we conclude that Corollary \ref{co1} holds.
		\end{proof}

		\begin{proof}[Proof of Theorem \ref{th3}]
			We proof Theorem \ref{th3} by Lemma \ref{leB3}. For all $k\in\{1,2,\dots,d\}$, $j\in\{1,\dots,k-1,k+1,\dots,d\}$, we define $\xi_{n,kj}=\ln (H_{nk}/H^{m_k/m_j}_{nj})$, $\Lambda_{n+1,kj}=\mathbb E(\xi_{n+1,kj}-\xi_{n,kj}|{\cal G}_n)$, and $W_{n+1,kj}=\mathbb E[(\xi_{n+1,kj}-\xi_{n,kj})^{2}|{\cal G}_n]$. Next, we will verify
			\begin{equation}\label{eqn4}
				\sum_n\Lambda_{n,kj}<\infty,\ \ \ \sum_n W_{n,kj}<\infty,\ \ a.s.
			\end{equation}
			respectively. We have
			\begin{eqnarray*}
				&&\xi_{n+1,kj}-\xi_{n,kj}=\ln \frac{H_{n+1,k}}{H^{m_k/m_j}_{n+1,j}}-\ln\frac{H_{nk}}{H^{m_k/m_j}_{nj}}\non\\
				&=&\ln H_{n+1,k}-\ln H_{nk}-\frac{m_k}{m_j}(\ln H_{n+1,j}-\ln H_{nj})\non\\
				&=&\ln(H_{nk}+A_{n+1,k}X_{n+1,k})-\ln 
				H_{nk}-\frac{m_k}{m_j}[\ln(H_{nj}+A_{n+1,j}X_{n+1,j})-\ln H_{nj}]\\
				&=&\int_{0}^{A_{n+1,k}X_{n+1,k}}\frac{1}{H_{nk}+s}\ ds-\frac{m_k}{m_j}\int_{0}^{A_{n+1,j}X_{n+1,j}}\frac{1}{H_{nj}+s}\ ds,
			\end{eqnarray*}
			For all $s_1\in(0,A_{n+1,k}X_{n+1,k})$ and $s_2\in(0,A_{n+1,j}X_{n+1,j})$, we have
			\begin{equation*}
				\frac{1}{H_{nk}}-\frac{s_1}{H^{2}_{nk}}\le \frac{1}{H_{nk}+s_1}\le \frac{1}{H_{nk}},\ \ \ \frac{1}{H_{nj}}-\frac{s_2}{H^{2}_{nj}}\le \frac{1}{H_{nj}+s_2}\le \frac{1}{H_{nj}},
			\end{equation*}
			Thus,
			\begin{eqnarray}\label{eqn5}
				\Lambda_{n,kj}&=&\mathbb E(\xi_{n+1,kj}-\xi_{n,kj}|{\cal G}_n)\non\\
				&\le&\mathbb E\left[ \left(\frac{A_{n+1,k}X_{n+1,k}}{H_{nk}}+\frac{m_k}{m_j}\frac{A^{2}_{n+1,j}X^{2}_{n+1,j}}{2H^{2}_{nj}}-\frac{m_k}{m_j}\frac{A_{n+1,j}X_{n+1,j}}{H_{nj}}\right)|{\cal G}_n\right]\non\\
				&=&\mathbb E\left(\frac{A_{n+1,k}N_{n+1}}{S_n}-\frac{m_k}{m_j}\frac{A_{n+1,j}N_{n+1}}{S_n}|{\cal G}_n\right)+\frac{m_k}{m_j}\mathbb E\left(\frac{A^{2}_{n+1,j}X^{2}_{n+1,j}}{2H^{2}_{nj}}|{\cal G}_n\right).
			\end{eqnarray}
			By (\ref{eqqx5}), the first term of (\ref{eqn5}) equal to
			\begin{equation}\label{eqn6}
				\frac{N_{n+1}}{S_n}\left(m_{n+1,k}-\frac{m_k}{m_j} m_{n+1,j}\right)\le\frac{C_1|m_jm_{n+1,k}-m_km_{n+1,j}|}{m_jn}.
			\end{equation}
			By Lemma \ref{leA3}, we have that for all 
			$0<\gamma_j<m_j/m_1$, $H^{-1}_{nj}=O_{a.s.}(n^{-\gamma_j})$. Thus, there exist constants $C_4$ such that the second term of (\ref{eqn5}) less than or equals to
			\begin{eqnarray}\label{eqn7}
				\frac{C_1m_k}{2m_j} \frac{\mathbb E(A^{2}_{n+1,j})}{H_{nj}S_n}N_{n+1}\le \frac{C^2_1C_2m_k}{2m_j} \frac{1}{H_{nj}S_n}\le C_4 \frac{1}{n^{1+\gamma_j}},\ \ a.s..
			\end{eqnarray}
			Then, by Assumption \ref{as3} and (\ref{eqn5})-(\ref{eqn7}), we conclude
			\begin{equation*}
				\sum\limits_{n=1}^{\infty}\Lambda_{n,kj}<\infty,\ \ a.s..
			\end{equation*}
			The first part of (\ref{eqn4}) is proved, for the second part, we have
			\begin{eqnarray*}
				W_{n+1,kj}&=&\mathbb E[(\xi_{n+1,kj}-\xi_{n,kj})^{2}|{\cal G}_n]\\
				&=&\mathbb E\left\{\left[\ln H_{n+1,k}-\ln H_{nk}-\frac{m_k}{m_j}(\ln H_{n+1,j}-\ln H_{nj})\right]^{2}|{\cal G}_n\right\}\\
				&\le&2\left\{\mathbb E[(\ln H_{n+1,k}-\ln H_{nk})^{2}|{\cal G}_n]+\frac{m^{2}_k}{m^{2}_j}\mathbb E[(\ln H_{n+1,j}-\ln H_{nj})^{2}|{\cal G}_n]\right\}\\
				&\le&2\mathbb E\left[\left(\frac{A^{2}_{n+1,k}X^{2}_{n+1,k}}{H^{2}_{nk}}+\frac{m^{2}_k}{m^{2}_j}\frac{A^{2}_{n+1,j}X^{2}_{n+1,j}}{H^{2}_{nj}}\right)|{\cal G}_n\right].
			\end{eqnarray*}
			The form of the above equation is similar to the second term of (\ref{eqn5}). From (\ref{eqn7}), we conclude		
			$$\sum\limits_{n=1}^{\infty}W_{n,kj}<\infty,\ \ a.s..$$
			Thus, by Lemma \ref{leB3}, we obtain that for all $k\in\{1,2,\dots,d\}$, $j\in\{1,\dots,k-1,k+1,\dots,d\}$, there exists a finite random variable $\tilde{\xi}_{kj}$ such that 
			$$\xi_{n,kj}=\ln \frac{H_{nk}}{H^{m_k/m_j}_{nj}}\overset{a.s.}\to\tilde{\xi}_{kj},$$
			then, $\frac{H_{nk}}{H^{m_k/m_j}_{nj}}$ converges almost surely to a finite random variable $\xi_{kj}$ on $(0,+\infty)$. Theorem \ref{th3} is proved.
		\end{proof}

		\begin{proof}[Proof of Corollary \ref{co2}]
			When $t=1$, it is obviously that $Z_1$ equals to 1 almost surely by Theorem \ref{th2}. Next we consider the case where $t>1$. Note that
			\begin{eqnarray*}
				\frac{1}{Z_{nk}}=\frac{S_n}{H_{nk}}=\frac{H_{n1}}{H_{nk}}+\dots+\frac{H_{n,k-1}}{H_{nk}}+1+\frac{H_{n,k+1}}{H_{nk}}+\dots+\frac{H_{n,t+1}}{H_{nk}}+\dots+\frac{H_{nd}}{H_{nk}},
			\end{eqnarray*}
			By Theorem \ref{th3}, we have that for all $s_1\in\{1,\dots,k-1,k+1,\dots,t\}$ and  $s_2\in\{t+1,\dots,d\}$,
			\begin{equation*}
				\frac{H_{ns_1}}{H_{nk}}\overset{a.s.}{\to}\xi_{sk}\in(0,+\infty),\quad \frac{H_{ns_2}}{H_{nk}}\overset{a.s.}{\to}0.
			\end{equation*}
			Thus, we have 
			\begin{equation*}
				{\mathbb P}\left(1\le\lim_{n\to\infty}\left\{\frac{H_{n1}}{H_{nk}}+\dots+\frac{H_{n,k-1}}{H_{nk}}+1+\frac{H_{n,k+1}}{H_{nk}}+\dots+\frac{H_{n,t+1}}{H_{nk}}+\dots+\frac{H_{nd}}{H_{nk}}\right\}<+\infty\right)=1,
			\end{equation*}
			then, we obtain
			\begin{equation*}
				{\mathbb P}(0<Z_k\le1)={\mathbb P}(0<\lim_{n\to\infty}Z_{nk}\le1)=1, 
			\end{equation*}
			which indicates that $\mathbb P(Z_k=0)=1$.
			And, for all $k\in\{1,2,\dots,t\}$, we have
			\begin{equation*}
				{\mathbb P}(Z_k=1)={\mathbb P}\left(\cap_{j\neq k}\{Z_j=0\}\right)\le \min\{{\mathbb P}(Z_1=0), {\mathbb P}(Z_2=0)\}=0.
			\end{equation*}
		\end{proof}

		\begin{proof}[Proof of Theorem \ref{th6}]
			We first prove (i) of Theorem \ref{th6}. For all $j\in\{t+1,\dots,d\}$, by Corollary \ref{co1}, we have that $Z_{n1}=H_{n1}/S_n$ and $S_n/n$ converges almost surely to  $Z_1$ and $Nm_1$, respectively. By Corollary \ref{co2}, we have $\mathbb P(Z_1=0)=0$. And by Theorem \ref{th3}, we have $H_{nj}/H^{m_j/m_1}_{n1}$ converges almost surely to $\xi_{j1}$ with $\mathbb P(\xi_{j1}\in(0,\infty))=1$. Then, we obtain
			\begin{eqnarray*}
				n^{1-m_j/m_1}Z_{nj}=\frac{H_{nj}}{H^{m_j/m_1}_{n1}}\frac{H^{m_j/m_1}_{n1}}{S^{m_j/m_1}_n}\frac{n^{1-m_j/m_1}}{S^{1-m_j/m_1}_n}
				\overset{a.s.}{\to}\frac{\xi_{j1}Z^{m_j/m_1}_1}{\left(Nm_1\right)^{1-m_j/m_1}}\overset{a.s.}{=}\tilde{Z}_j\in(0,+\infty).
			\end{eqnarray*}
			
			For (ii) of Theorem \ref{th6}, we have that for all $j\in\{t+1,\dots,d\}$, 
			$$\frac{H_{nj}}{n^{m_j/m_1}}=\frac{H_{nj}}{H^{m_j/m_1}_{n1}}\frac{H^{m_j/m_1}_{n1}}{n^{m_j/m_1}}\overset{a.s.}{\to}(Nm_1Z_1)^{m_j/m_1}\xi_{j1}.$$
			
			For (iii) of Theorem \ref{th6}, we choose $p=2$, $\alpha_j=j^{1-m_k/m_1}$, $\beta_j=j^{m_k/m_1}$ and $Y_j=j^{1-m_k/m_1}X_{jk}$ in Lemma \ref{leB1}. Then, we have that when $n$ tends to $\infty$,
			\begin{eqnarray*}
				\beta_n\uparrow+\infty,\ \  \frac{1}{\beta_{n}}\sum\limits_{j=1}^{n}\frac{1}{\alpha_j}=\frac{1}{n^{m_k/m_1}}\sum\limits_{j=1}^{n}j^{m_k/m_1-1}\approx\frac{1}{n^{m_k/m_1}}\int_{1}^{n} x^{m_k/m_1-1} dx\to \frac{m_1}{m_k},
			\end{eqnarray*}
			and
			\begin{eqnarray*}
				\mathbb E(Y_j|{\cal F}_{j-1})=j^{1-m_k/m_1}\mathbb E(N_j|{\cal F}_{j-1})Z_{j-1,k}\overset{a.s.}{\to}
				\begin{cases}
					NZ_k, &\text {$k\in\{1,\dots,t\}$},\\
					N\tilde{Z}_k, &\text{ $k\in\{t+1,\dots,d\}$}.
				\end{cases}
			\end{eqnarray*}
			where for two functions $f(n)$ and $g(n)$ with respect to $n$, the expression $f(n)\approx g(n)$ denotes that $\lim\limits_{n\to\infty}f(n)/g(n)=1$. Moreover, we have
			\begin{eqnarray*}
				\sum_{j=1}^{\infty}\frac{\mathbb E(Y^2_j|{\cal F}_{j-1})}{(\alpha_j\beta_j)^2}=\sum_{j=1}^{\infty}\frac{j^{2-2m_k/m_1}\mathbb E(X^2_{jk}|{\cal F}_{j-1})}{j^2}=\sum\limits_{j=1}^{\infty}\frac{j^{1-m_k/m_1}\mathbb E(X^2_{jk}|{\cal F}_{j-1})}{j^{1+m_k/m_1}}.
			\end{eqnarray*}
			Note that
			\begin{eqnarray}\label{eqn13}
				\mathbb E(X^2_{jk}|{\cal F}_{j-1})&=&\mathbb E\left[\frac{N_jH_{j-1,k}(N_jH_{j-1,k}-N_j-H_{j-1,k}+S_{j-1})}{S_{j-1}(S_{j-1}-1)}|{\cal F}_{j-1}\right]\non\\	
				&\approx&\mathbb E(N^2_j-N_j|{\cal F}_{j-1})Z^2_{j-1,k}+Z_{j-1,k}\mathbb E\left(N_j|{\cal F}_{j-1}\right),
			\end{eqnarray}
			Then, we have
			\begin{equation*}
				j^{1-m_k/m_1}\mathbb E(X^2_{jk}|{\cal F}_{j-1})\overset{a.s.}{\to}\begin{cases}
					(Q-N)Z^2_k+NZ_k, &\text{$k\in\{1,\dots,t\}$},\\
					N\tilde{Z}_k, &\text{$k\in\{t+1,\dots,d\}$}.
				\end{cases}
			\end{equation*}
			Recall that $Z_k(k=1,\dots,t)$ is bounded, and for all $k\in\{t+1,\dots,d\}$, $\mathbb P(\tilde{Z}_{k}\in(0,+\infty)=1$, then we obtain
			\begin{equation*}
				\sum_{j=1}^{\infty}\frac{\mathbb E(Y^2_j|{\cal F}_{j-1})}{(\alpha_j\beta_j)^2}<+\infty,\ \ a.s..
			\end{equation*}
			Thus, by Lemma \ref{leB1}, we conclude
			\begin{equation*}
				\frac{N_{A_k,n}}{n^{m_k/m_1}}=\frac{1}{n^{m_k/m_1}}\sum\limits_{j=1}^{n}X_{jk}\overset{a.s.}{\to}\begin{cases}
					NZ_k, &\text{$k\in\{1,\dots,t\}$}\\
					\frac{m_1}{m_k}N\tilde{Z}_k, &\text{$k\in\{t+1,\dots,d\}$}.
				\end{cases}
			\end{equation*}
			Theorem \ref{th6} is proved.
		\end{proof}

		\begin{proof}[Proof of Theorem \ref{th4}]
			The almost sure convergence of $\hat{\mu}_n$, $\hat{q}_{N,n}$ and $\hat\nu_{nk}$ can be obtained directly from Lemma \ref{leB1}, where we choose $\alpha_j=1$ and $\beta_j=j$. Then, we prove the strong consistency of $\hat{m}_{A_k,n}$, $\hat{q}_{A_k,n}$ and $\tilde{q}_{A_kA_s,n}$ for $k,s$ such that $m_k+m_s>m_1$. We have
			\begin{eqnarray}\label{eqn11}
				\frac{1}{n^{m_k/m_1}}\sum\limits_{j=1}^{n}A_{jk}X_{jk}\overset{a.s.}{\to}
				\begin{cases}
					m_1NZ_k, &\text{$k\in\{1,\dots,t\}$},\\
					m_1N\tilde{Z}_k, &\text{$k\in\{t+1,\dots,d\}$}.
				\end{cases}
			\end{eqnarray}
			\begin{eqnarray}\label{eqn12}
				\frac{1}{n^{m_k/m_1}}\sum\limits_{j=1}^{n}A^2_{jk}X_{jk}\overset{a.s.}{\to}
				\begin{cases}
					q_kNZ_k, &\text{$k\in\{1,\dots,t\}$},\\
					\frac{m_1q_k}{m_k}N\tilde{Z}_k, &\text{$k\in\{t+1,\dots,d\}$}.
				\end{cases}
			\end{eqnarray}
			and 
			\begin{eqnarray}\label{eqnx12}
				\frac{1}{n^{m_k/m_1}}\sum\limits_{j=1}^{n}A_{jk}A_{js}X_{jk}\overset{a.s.}{\to}
				\begin{cases}
					q_{ks}NZ_k, &\text{$k\in\{1,\dots,t\}$},\\
					\frac{m_1q_{ks}}{m_k}N\tilde{Z}_k, &\text{$k\in\{t+1,\dots,d\}$}.
				\end{cases}
			\end{eqnarray}
			respectively. The conclusions (\ref{eqn11}), (\ref{eqn12}) and (\ref{eqnx12}) can be proved similarly to (\ref{eqn10}), and it is worth noting that the condition of finite 2+$\epsilon$ order moments of the replacement matrices are needed when verifying the conditions in Lemma \ref{leB1}, the details are omitted. Thus, the strong consistency of $\hat{m}_{A_k,n}$, $\hat{q}_{A_k,n}$ and $\tilde{q}_{A_kA_s,n}$ for $k,s$ such that $m_k+m_s>m_1$ are obtained by (\ref{eqn10}), (\ref{eqn11}), (\ref{eqn12}) and (\ref{eqnx12}).
			
			Next, we consider the strong consistency of $\hat{q}_{A_kA_s,n}$ for $k,s$ such that $m_k+m_s>m_1$. We will prove
			\begin{equation}\label{eqn14}
				\frac{1}{n^{(m_k+m_s-m_1)/m_1}}\sum\limits_{j=1}^{n}X_{jk}X_{js}\overset{a.s.}{\to}\begin{cases}
					(Q-N)Z_kZ_s,\ &\text{$1\le k<s\le t,$}\\
					\frac{m_1}{m_s}(Q-N)Z_k\tilde{Z}_s,\ &\text{$1\le k\le t<s\le d.$}\\
					\frac{m_1}{m_k+m_s-m_1}(Q-N)\tilde{Z}_k\tilde{Z}_s,\ &\text{$1\le t< k<s\le d.$}
				\end{cases}
			\end{equation}
			and
			\begin{equation}\label{eqn15}
				\frac{1}{n^{(m_k+m_s-m_1)/m_1}}\sum\limits_{j=1}^{n}A_{jk}A_{js}X_{jk}X_{js}\overset{a.s.}{\to}\begin{cases}
					(Q-N)q_{ks}Z_kZ_s,\ &\text{$1\le k<s\le t,$}\\
					\frac{m_1}{m_s}(Q-N)q_{ks}Z_k\tilde{Z}_s,\ &\text{$1\le k\le t<s\le d.$}\\
					\frac{m_1}{m_k+m_s-m_1}(Q-N)q_{ks}\tilde{Z}_k\tilde{Z}_s,\ &\text{$1\le t< k<s\le d.$}
				\end{cases}
			\end{equation}
			separately. Note that
			\begin{eqnarray*}
				\mathbb E(X_{jk}X_{js}|{\cal F}_{j-1})&=&-\mathbb E\left[\frac{N_j(S_{j-1}-N_j)}{S^2_{j-1}(S_{j-1}-1)}H_{j-1,k}H_{j-1,s}|{\cal F}_{j-1}\right]+\mathbb E(N^2_j|{\cal F}_{j-1})Z_{j-1,k}Z_{j-1,s}\\
				&&\approx\mathbb E(N^2_j-N_j|{\cal F}_{j-1})Z_{j-1,k}Z_{j-1,s}.
			\end{eqnarray*}
			For $1\le k<s\le t$, we have 
			\begin{equation}\label{eqn16}
				\frac{1}{n}\sum\limits_{j=1}^{n}X_{jk}X_{js}\overset{a.s.}{\to}(Q-N)Z_kZ_s.
			\end{equation}
			For $ t<s\le d$, we choose 
			\begin{equation*}
				Y_j=j^{1-m_k/m_1}j^{1-m_s/m_1}X_{jk}X_{js},\ \alpha_j=j^{2-m_k/m_1-m_s/m_1},\  \mbox{and}\ \  \beta_j=j^{(m_k+m_s-m_1)/m_1}
			\end{equation*}
			in Lemma \ref{leB1}, since $m_k+m_s>m_1$, we obtain
			\begin{align*}
				\frac{1}{\beta_n}\sum\limits_{j=1}^{n}\frac{1}{\alpha_j}&\to\begin{cases}
					\frac{m_1}{m_s},\ &\text{$1\le k\le t<s\le d,$}\\
					\frac{m_1}{m_k+m_s-m_1}, &\text{$1\le t<k<s\le d.$}
				\end{cases}\\
				\mathbb E(Y_j|{\cal F}_{j-1})&\overset{a.s.}{\to}
				\begin{cases}
					(Q-N)Z_k\tilde{Z}_s,\ &\text{$1\le k\le t<s\le d,$}\\
					(Q-N)\tilde{Z}_k\tilde{Z}_s,\ &\text{$1\le t<k<s\le d.$}
				\end{cases}
			\end{align*}
			and
			\begin{eqnarray*}
				\sum_{j=1}^{\infty}\frac{\mathbb E(Y^2_j|{\cal F}_{j-1})}{(\alpha_j\beta_j)^2}&=&\sum\limits_{j=1}^\infty \frac{j^{4-2(m_k+m_s)/m_1}\mathbb E(X^2_{jk}X^2_{js}|{\cal F}_{j-1})}{j^2}\\
				&\le& C^2_1\sum\limits_{j=1}^\infty\frac{ j^{2-(m_k+m_s)/m_1}\mathbb E(X_{jk}X_{js}|{\cal F}_{j-1})}{j^{(m_k+m_s)/m_1}} <+\infty,\ \ a.s..
			\end{eqnarray*}
			Thus, when $m_k+m_s>m_1$, we obtain
			\begin{equation}\label{eqn17}
				\frac{1}{n^{(m_k+m_s-m_1)/m_1}}\sum\limits_{j=1}^{n}X_{jk}X_{js}\overset{a.s.}{\to}\begin{cases}
					\frac{m_1}{m_s}(Q-N)Z_k\tilde{Z}_s,\ &\text{$1\le k\le t<s\le d.$}\\
					\frac{m_1}{m_k+m_s-m_1}(Q-N)\tilde{Z}_k\tilde{Z}_s,\ &\text{$1\le t< k<s\le d.$}
				\end{cases}
			\end{equation}
			Equation (\ref{eqn14}) is obtained by (\ref{eqn16}) and (\ref{eqn17}). Equation (\ref{eqn15}) can be obtained by the same method. Then, Theorem \ref{th4} is proved.
		\end{proof}

		\begin{proof}[Proof of Theorem \ref{thm5}]
			Define ${\bf Y}_{ni}=\left(\frac{X_{i1}(A_{i1}-m_{i1})}{\sqrt{n}},
			\dots,\frac{X_{id}(A_{id}-m_{id})}{\sqrt{n^{m_d/m_1}}}\right)^{\top}$, 
			we first prove that there exists a symmetric random matrix $\tilde{\Sigma}$ such that
			\begin{equation}\label{eqn18}
				\sum\limits_{i=1}^{n}{\bf Y}_{ni}\overset{stable}{\to}{\cal N}(0,\tilde{\Sigma}).
			\end{equation}
			We set ${\cal F}_{n,i-1}={\cal F}_{i-1}$. Note that $\{{\bf Y}_{ni},i=1,\dots,n\}$ is an array of martingale differences. By Corollary 3.1 of \cite{r46}, it is sufficient to show that for all $\epsilon>0$,
			\begin{equation}\label{eqn19}
				\sum\limits_{i=1}^{n}\mathbb E[\|{\bf Y}_{ni}\|^2I\{\|{\bf Y}_{ni}\|\ge \epsilon\}|{\cal F}_{n,i-1}]\overset{p}{\to}0,\ \ \ {\rm and}
			\end{equation}
			\begin{equation}\label{eqn20}
				\sum\limits_{i=1}^{n}\mathbb E({\bf Y}_{ni}^{T}{\bf Y}_{ni}|{\cal F}_{n,i-1})\overset{{\mathbb P}}{\to}\tilde{\Sigma}.
			\end{equation}
			For (\ref{eqn19}), we have
			\begin{eqnarray*}
				&&\sum\limits_{i=1}^{n}\mathbb E[\|{\bf Y}_{ni}\|^2 I(\|{\bf Y}_{ni}\|\ge\epsilon)|{\cal F}_{n,i-1}]\\
				&=&\sum\limits_{j=1}^{d}\sum\limits_{i=1}^{n}\mathbb E[Y^{2}_{ni,j} I(\|{\bf Y}_{ni}\|\ge\epsilon)|{\cal F}_{n,i-1}]
				\le\sum\limits_{j=1}^{d}\sum\limits_{k=1}^{d}\sum\limits_{i=1}^{n}\mathbb E[Y^{2}_{ni,j}  I(|Y_{ni,k}|\ge\epsilon/\sqrt{d})|{\cal F}_{n,i-1}],
			\end{eqnarray*}
			thus, it remains only to show that $j,k\in\{1,2,\dots,d\}$,
			\begin{equation}\label{eqn21}
				\sum\limits_{i=1}^{n}\mathbb E[Y^{2}_{ni,j}  I(|Y_{ni,k}|\ge\epsilon/\sqrt{d})|{\cal F}_{n,i-1}]\overset{p}{\to}0.
			\end{equation}
			Define $\tilde\epsilon=\epsilon/(C_1\sqrt{d})$, then we have
			\begin{align}\label{eqn22}
				&\sum\limits_{i=1}^{n}\mathbb E[Y^{2}_{ni,j}  I(|Y_{ni,k}|\ge\epsilon/\sqrt{d})|{\cal F}_{n,i-1}]\non\\
				\le&\frac{1}{n^{m_j/m_1}}\sum\limits_{i=1}^{n}\mathbb E(X^2_{ij}|{\cal F}_{n,i-1})\mathbb E\left\{(A_{ij}-m_{ij})^2I\left[|A_{ik}-m_{ik}|\ge \tilde\epsilon\sqrt{n^{m_k/m_1}}\right]\right\}\non\\
				\le&\frac{1}{n^{m_j/m_1}}\sum\limits_{i=1}^{n}\mathbb E(X^2_{ij}|{\cal F}_{n,i-1})\sup_{i\ge1}\mathbb E\left\{(A_{ij}-m_{ij})^2I\left[|A_{ik}-m_{ik}|\ge \tilde\epsilon\sqrt{n^{m_k/m_1}}\right]\right\}.
			\end{align}
			By (\ref{eqn13}), we have that for all $j\in\{1,\dots,d\}$,
			\begin{equation*}\label{eqn23}
				\frac{1}{n^{m_j/m_1}}\sum\limits_{i=1}^{n}\mathbb E(X^2_{ij}|{\cal F}_{n,i-1})=\frac{1}{n^{m_j/m_1}}\sum\limits_{i=1}^{n}[\mathbb E(N^2_i-N_i|{\cal F}_{n,i-1})Z^2_{i-1,j}+Z_{i-1,j}\mathbb E(N_i|{\cal F}_{n,i-1})].
			\end{equation*}
			Since that for all $1\le j\le t$, $Z_{i-1,j}\overset{a.s.}{\to}{Z_j}$ when $i$ tends to $\infty$, by Ces\`{a}ro summation, we have
			\begin{equation*}
				\frac{1}{n^{m_j/m_1}}\sum\limits_{i=1}^{n}\mathbb E(X^2_{ij}|{\cal F}_{n,i-1})\overset{a.s.}{\to}(Q-N)Z^2_j+NZ_j.
			\end{equation*}
			When $t<j\le d$, by Theorem \ref{th6}, we have $i^{1-m_j/m_1}Z_{i-1,j}\overset{a.s.}{\to}\tilde{Z}_j$ when $i$ tends to $\infty$. In Lemma \ref{leB2}, we choose 
			\begin{equation*}
				Y_i=i^{1-m_j/m_1}Z_{i-1,j}\mathbb E(N_i|{\cal F}_{n,i-1})\overset{a.s.}{\to}N\tilde{Z}_j,\ a_i=i^{m_j/m_1-1},\ b_n=\sum\limits_{i=1}^{n}a_i\approx \frac{m_1}{m_j}n^{m_j/m_1},
			\end{equation*}
			then we have
			\begin{equation}\label{eqn24}
				\frac{1}{n^{m_j/m_1}}\sum\limits_{i=1}^{n}\mathbb E(X^2_{ij}|{\cal F}_{n,i-1})\approx	\frac{1}{n^{m_j/m_1}}\sum\limits_{i=1}^{n}Z_{i-1,j}\mathbb E(N_i|{\cal F}_{i-1})\overset{a.s.}{\to}\frac{m_1}{m_q}N\tilde{Z}_j,
			\end{equation}
			Note that 
			\begin{equation}\label{eqn25}
				\sup_{i\ge1}\mathbb E\left\{(A_{ij}-m_{ij})^2I\left[|A_{ik}-m_{ik}|\ge \tilde\epsilon\sqrt{n^{m_k/m_1}}\right]\right\}\to0,
			\end{equation}
			by (\ref{eqn22})--(\ref{eqn25}), we conclude (\ref{eqn21}) holds, it follows that (\ref{eqn19}) holds. 
			
			Next, we consider (\ref{eqn20}). Define ${\bf Q}_{ni}=\sum_{i=1}^{n}\mathbb E({\bf Y}_{ni}^{T}{\bf Y}_{ni}|{\cal F}_{n,i-1})$, we consider separately the diagonal and non-diagonal elements of ${\bf Q}_{ni}$. By (\ref{eqn13}), the $q$th diagonal element of ${\bf Q}_{ni}$ equals to
			\begin{eqnarray*}
				({\bf Q}_{ni})_{q,q}&=&\frac{1}{n^{m_q/m_1}}\sum\limits_{i=1}^{n}\mathbb E(X^2_{iq}|{\cal F}_{n,i-1})\mathbb E(A_{iq}-m_{iq})^2\\
				&\approx&\frac{1}{n^{m_q/m_1}}\sum\limits_{i=1}^{n}\sigma^2_{q,i}\left[\mathbb E(N^2_i-N_i|{\cal F}_{n,i-1})Z^2_{i-1,q}+Z_{i-1,q}\mathbb E(N_i|{\cal F}_{n,i-1})\right],
			\end{eqnarray*}
			When $q=1,\dots,t$, by Theorem \ref{th1}, we have $Z_{i-1,q}$ converges to $Z_q$ almost surely when $i$ tends to infty. Note that $\sigma^2_{q,i}$ converges to $\sigma^2_{q}$, thus, by Ces\`{a}ro summation, we have
			\begin{equation}\label{eqn26}
				({\bf Q}_{ni})_{q,q}\overset{a.s.}{\to}[(Q-N)Z^2_q+NZ_q]\sigma^2_q,\ \ q=1,\dots,t.
			\end{equation}
			When $q=t+1,\dots,d$, in Lemma \ref{leB2}, we choose
			\begin{equation*}
				Y_i=i^{1-m_q/m_1}Z_{i-1,q}\mathbb E(N_i|{\cal F}_{n,i-1})\sigma^2_{q,i}\overset{a.s.}{\to}N\tilde{Z}_q\sigma^2_q,\ a_i=i^{m_q/m_1-1},\ b_n=\sum\limits_{i=1}^{n}a_i\approx \frac{m_1}{m_q}n^{m_q/m_1},
			\end{equation*}
			then we have
			\begin{equation}\label{eqn27}
				({\bf Q}_{ni})_{q,q}\approx\frac{1}{n^{m_q/m_1}}\sum\limits_{i=1}^{n}Z_{i-1,q}\mathbb E(N_i|{\cal F}_{n,i-1})\sigma^2_{q,i}\overset{a.s.}{\to}\frac{m_1}{m_q}N\tilde{Z}_q\sigma^2_q,\ q=t+1,\dots,d.
			\end{equation}
			
			Next we consider the nondiagonal elements of ${\bf Q}_{ni}$. For $1\le p<q\le d$,
			\begin{eqnarray*}
				&&({\bf Q}_{ni})_{p,q}=\sum\limits_{i=1}^{n}\mathbb E\left(\frac{X_{ip}X_{iq}(A_{ip}-m_{ip})(A_{iq}-m_{iq})}{\sqrt{n^{m_p/m_1}n^{m_q/m_1}}}|{\cal F}_{n,i-1}\right)\\
				&=&\frac{1}{\sqrt{n^{m_p/m_1}n^{m_q/m_1}}}\sum\limits_{i=1}^{n}c_{pq,i}\mathbb E(X_{ip}X_{iq}|{\cal F}_{n,i-1})\\
				&=&\frac{1}{\sqrt{n^{m_p/m_1}n^{m^q/m_1}}}\sum\limits_{i=1}^{n}c_{pq,i}\mathbb E\left[-\frac{N_iH_{i-1,p}H_{i-1,q}(S_{i-1}-N_i)}{S^2_{i-1}(S_{i-1}-1)}+N^2_iZ_{i-1,p}Z_{i-1,q}|{\cal F}_{n,i-1}\right]\\
				&\approx&\frac{1}{\sqrt{n^{m_p/m_1}n^{m_q/m_1}}}\sum\limits_{i=1}^{n}c_{pq,i}\mathbb E(-N_iZ_{i-1,p}Z_{i-1,q}+N^2_iZ_{i-1,p}Z_{i-1,q}|{\cal F}_{n,i-1})\\
				&=&\frac{1}{\sqrt{n^{m_p/m_1}n^{m_q/m_1}}}\sum\limits_{i=1}^{n}c_{pq,i}Z_{i-1,p}Z_{i-1,q}\mathbb E(N^2_i-N_i|{\cal F}_{n,i-1}),
			\end{eqnarray*}
			When $1\le p<q\le t\le d$, by Ces\`{a}ro summation, we have
			\begin{equation}\label{eqn28}
				({\bf Q}_{ni})_{p,q}\overset{a.s.}{\to}c_{pq}Z_pZ_q(Q-N).
			\end{equation}
			When $t<q\le d$, we have
			\begin{equation}\label{eqn29}
				({\bf Q}_{ni})_{p,q}\approx \sqrt{\frac{n^{m_q/m_1}}{n^{m_p/m_1}}}\frac{1}{n^{mq/m_1}}\sum\limits_{i=1}^{n}c_{pq,i}Z_{i-1,p}Z_{i-1,q}\mathbb E(N^2_i-N_i|{\cal F}_{n,i-1}).
			\end{equation}
			Note that when $1\le p\le t<q\le d$, we have $\sqrt{\frac{n^{m_q/m_1}}{n^{m_p/m_1}}}\to0$. In Lemma \ref{leB2}, we choose
			\begin{equation*}
				a_{i}=i^{m_q/m_1-1},\ \  b_n=\sum\limits_{i=1}^{n}a_i\approx \frac{m_1}{m_q}n^{m_q/m_1},
			\end{equation*}
			and
			\begin{equation*}
				Y_i=i^{1-m_q/m_1}c_{pq,i}Z_{i-1,p}Z_{i-1,q}\mathbb E(N^2_i-N_i|{\cal F}_{i-1})\overset{a.s.}{\to}\begin{cases}
					\frac{m_1}{m_q}c_{pq}Z_p\tilde{Z}_q(Q-N),\ &\text{$1\le p\le t<q\le d,$}\\
					0,  &\text{$1\le t<p<q\le d.$}
				\end{cases}
			\end{equation*}
			then we obtain
			\begin{equation}\label{eqn31}
				({\bf Q}_{ni})_{p,q}\overset{a.s.}{\to}0,\ t<q\le d.
			\end{equation}
			By (\ref{eqn26})--(\ref{eqn31}), we conclude (\ref{eqn18}) holds. Moreover, the covariance matrix is
			$$\tilde{\Sigma}=    \footnotesize{
				\begin{pmatrix}
					\sigma^{2}_1[Z^2_1(Q-N)+Z_1N]&\dots&c_{1t}Z_1Z_t(Q-N)&0&\dots&0\\
					\vdots&\ddots&\vdots&\vdots&\ddots&\vdots\\
					c_{1t}Z_1Z_t(Q-N)&\dots&\sigma^{2}_t[Z^2_t(Q-N)+Z_tN]&0&\dots&0\\
					0&\dots&0&\frac{m_1}{m_{t+1}}\sigma^{2}_{t+1}\tilde{Z}_{t+1}N&\dots&0\\
					\vdots&\ddots&\vdots&\vdots&\ddots&\vdots\\
					0&\dots&0&0&\dots&\frac{m_1}{m_d}\sigma^{2}_{d}\tilde{Z}_{d}N
			\end{pmatrix}}.$$
			Recall the definition of $N_{A_k,n}$ and $\hat{m}_{A_k,n}(k=1,\dots,d)$, we have
			\begin{align}\label{eqn32}
				&\left(\sqrt{N_{A_1,n}}(\hat{m}_{A_1,n}-m_1),\dots,\sqrt{N_{A_d,n}}(\hat{m}_{A_d,n}-m_d)\right)^{\top}\non\\
				=&{\rm diag}\left\{\sqrt{\frac{n}{N_{A_1,n}}},\dots,\sqrt{\frac{n^{m_d/m_1}}{N_{A_d,n}}}\right\}\sum\limits_{i=1}^{n}{\bf Y}_{ni}+\left(\frac{\sum_{i=1}^{n}(m_{i1}-m_1)X_{i1}}{\sqrt{N_{A_1,n}}},\dots,\frac{\sum_{i=1}^{n}(m_{id}-m_d)X_{id}}{\sqrt{N_{A_d,n}}}\right)^{\top}.
			\end{align}
			By (\ref{eqn10}), we have 
			\begin{equation}\label{eqn33}
				\sqrt{\frac{n^{m_k/m_1}}{N_{A_k,n}}}\overset{a.s.}{\to} \begin{cases}
					\frac{1}{\sqrt{NZ_k}},\ &\text{$k\in\{1,\dots,t\}$},\\
					\frac{1}{\sqrt{m_1/m_kN\tilde{Z}_k}},\ &\text{$k\in\{t+1,\dots,d\}$}.
				\end{cases}
			\end{equation}
			For all $k\in\{1,2,\dots,d\}$,
			\begin{equation}\label{eqn34}
				\frac{\sum_{i=1}^{n}(m_{ik}-m_k)X_{ik}}{\sqrt{N_{A_k,n}}}=\frac{\frac{1}{n^{m_k/2m_1}}\sum_{i=1}^{n}(m_{ik}-m_k)X_{ik}}{\sqrt{\frac{N_{A_k,n}}{n^{m_k/m_1}}}},
			\end{equation}
			For the term $\frac{1}{n^{m_k/2m_1}}\sum_{i=1}^{n}(m_{ik}-m_k)X_{ik}$, in Lemma \ref{leB1}, we choose 
			\begin{equation*}
				\alpha_i=i^{m_k/(2m_1)-1},\ \ \beta_n=\sum_{i=1}^n\alpha_i\approx 2m_1/m_kn^{m_k/(2m_1)}\uparrow\infty,
			\end{equation*}
			and by Assumption \ref{as4}, we have
			\begin{equation*}
				Y_i=i^{1-m_k/(2m_1)}(m_{ik}-m_k)X_{ik},\ \ \mbox{then}\ \ \mathbb E(Y_i|{\cal F}_{i-1})=o(i^{1-m_k/m_1}Z_{i-1,k})\overset{a.s.}{\to}0.
			\end{equation*}
			Note that for all $p>1$, we have
			$$\sum_{i=1}^{n}\mathbb E\left(\frac{|Y_i|^p}{(\alpha_i\beta_i)^p}|{\cal F}_{i-1}\right)=O_{a.s.}(\sum_{i=1}^{n}i^{(1-p)m_k/m_1-1})<\infty,\ \ a.s.,\ \ \mbox{as}\ n\ \mbox{tends\ to}\ \infty.$$
			Then, we conclude
			\begin{equation}\label{eqn35}
				\frac{1}{n^{m_k/2m_1}}\sum_{i=1}^{n}(m_{ik}-m_k)X_{ik}\overset{a.s.}{\to}0.
			\end{equation}
			Thus, by Slutsky's Theorem and equations (\ref{eqn18}) and (\ref{eqn32})-(\ref{eqn35}), we conclude that Theorem \ref{thm5} holds.
		\end{proof}

		\begin{proof}[Proof of Corollary \ref{co4}]
			By Theorem \ref{thm5}, we have that
			\begin{equation*}
				\left(\sqrt{N_{A_k,n}}(\hat{m}_{A_k,n}-m_k),\sqrt{N_{A_{k+1},n}}(\hat{m}_{A_{k+1},n}-m_{k+1})\right)^{\top}\overset{stable}{\to}{\cal N}(0,\Sigma^{k,k+1}),
			\end{equation*}
			where $\Sigma^{k,k+1}$ is the submatrix consisting of the $(k,k+1)$st row and $(k,k+1)$st column of $\Sigma$ defined in Theorem \ref{thm5}. Note that
			\begin{eqnarray*}
				&&\hat{m}_{A_k,n}-\hat{m}_{A_{k+1},n}-(m_k-m_{k+1})\\
				&=&\left(1/\sqrt{N_{A_k},n},-1/\sqrt{N_{A_{k+1},n}}\right)\left(\sqrt{N_{A_k,n}}(\hat{m}_{A_k,n}-m_k),\sqrt{N_{A_{k+1},n}}(\hat{m}_{A_{k+1},n}-m_{k+1})\right)^{\top} 
			\end{eqnarray*}
			and 
			\begin{eqnarray*}
				&&\left(1/\sqrt{N_{A_k,n}},-1/\sqrt{N_{A_{k+1},n}}\right)\Sigma^{k,k+1}\left(1/\sqrt{N_{A_k,n}},-1/\sqrt{N_{A_{k+1},n}}\right)^{\top}\\
				&=&\left(1/\sqrt{N_{A_k,n}},-1\sqrt{N_{A_{k+1},n}}\right)\begin{pmatrix}  \Sigma^{k,k+1}_{11}&\Sigma^{k,k+1}_{12}\\ \Sigma^{k,k+1}_{21}& \Sigma^{k,k+1}_{22}\end{pmatrix}\left(1/\sqrt{N_{A_k,n}},-1/\sqrt{N_{A_{k+1},n}}\right)^{\top}\\
				&=&\begin{cases}
					\left(\frac{\sigma^2_k}{N_{A_k,n}}+\frac{\sigma^2_{k+1}}{N_{A_{k+1},n}}\right)T^{k,k+1}_n,\ &\text{$k+1\le t,$}\\
					\frac{\sigma^2_t[Z_t(Q/N-1)+1]}{N_{A_t,n}}+\frac{\sigma^2_{t+1}}{N_{A_{t+1},n}}\approx \frac{\sigma^2_t}{N_{A_t,n}}+\frac{\sigma^2_{t+1}}{N_{A_{t+1},n}},\ &\text{$k= t,$}\\
					\frac{\sigma^2_t}{N_{A_k,n}}+\frac{\sigma^2_{k+1}}{N_{A_{k+1},n}},\ &\text{$k>t.$}
				\end{cases}
			\end{eqnarray*}
			The Corollary \ref{co4} is proved.
		\end{proof}

		\begin{proof}[Proof of Corollary \ref{co5}]
			We define $\tilde\Sigma^{\ast}_n$ as a square matrix of $d-1$ dimensions, its diagonal elements are
			\begin{equation*}
				\left(\tilde\Sigma^{\ast}_n\right)_{q,q}=\frac{N_{A_{q+1},n}}{N_{A_q,n}}\Sigma_{qq}-2\sqrt{\frac{N_{A_{q+1},n}}{N_{A_q},n}}\Sigma_{q,q+1}+\Sigma_{q+1,q+1},
			\end{equation*}
			and the non-diagonal elements are
			\begin{eqnarray*}
				\left(\tilde\Sigma^{\ast}_n\right)_{p,q}=
				\sqrt{\frac{N_{A_{p+1},n}N_{A_{q+1},n}}{N_{A_p,n}N_{A_q,n}}}\Sigma_{pq}-\sqrt{\frac{N_{A_{q+1},n}}{N_{A_q,n}}}\Sigma_{p+1,q}-\sqrt{\frac{N_{A_{p+1},n}}{N_{A_p,n}}}\Sigma_{p,q+1}+\Sigma_{p+1,q+1}.
			\end{eqnarray*}
			By Theorem \ref{thm5}, we have 
			$$\left(\sqrt{N_{A_1,n}}(\hat{m}_{A_1,n}-m_1),\dots,\sqrt{N_{A_d,n}}(\hat{m}_{A_d,n}-m_d)\right)\overset{stable}\to 	\mathcal{N}(0,\Sigma).$$
			Define 
			$${\bf B}=\begin{pmatrix}
				\sqrt{\frac{N_{A_2,n}}{N_{A_1,n}}} & -1 & 0 & \cdots & 0 & 0\\
				0 & \sqrt{\frac{N_{A_3,n}}{N_{A_2,n}}} & -1 & \cdots & 0 & 0\\
				\vdots & \vdots & \vdots & \ddots & \vdots & \vdots\\
				0 & 0 & 0 & \cdots & \sqrt{\frac{N_{A_d,n}}{N_{A_{d-1},n}}} & -1
			\end{pmatrix}.$$
			Then, note that
			\begin{eqnarray*}
				&&\left(\sqrt{N_{A_2,n}}\left((\hat{m}_{A_1,n}-\hat{m}_{A_2,n})-(m_1-m_2)\right),\dots,\sqrt{N_{A_d,n}}\left((\hat{m}_{A_{d-1},n}-\hat{m}_{A_d,n})-(m_{d-1}-m_d) \right) \right)^{\top}\\
				&=&{\bf B}\left(\sqrt{N_{A_1,n}}(\hat{m}_{A_1,n}-m_1),\dots,\sqrt{N_{A_d,n}}(\hat{m}_{A_d,n}-m_d)\right)^{\top},
			\end{eqnarray*}
			and 
			${\bf B}\Sigma {\bf B}^{\top}=\tilde\Sigma^{\ast}_{n}$. Since $\tilde\Sigma^{\ast}_{n}-\Sigma^{\ast}_{n}\overset{a.s.}{\to}{\bf 0}$, we conclude Corollary \ref{co5} holds.
		\end{proof}

		\section{Technical results}\label{secA}
		In Lemma \ref{le1}, the number of balls of each type in the urn tends to infinity. In this section, we give some technical results, mainly to further obtain a higher order of $H_{nk}$ to inform its exact order.
		\begin{lemma}\label{leA2}
			Under Assumptions \ref{as1}-\ref{as2}, for all $k\in\{1,\dots,t\}$ and $j\in\{1,\dots,k-1,k+1,\dots,d\}$, we have that for all $\theta_j> \frac{m_1}{m_j}$, 
			$$\frac{H_{nk}}{H^{\theta_j}_{nj}}\overset{a.s.}{\to}0.$$
		\end{lemma}
		\begin{proof}
			Define $R_{n}=H_{nk}/H^{\theta_j}_{nj}$, then for all $\theta_j> m_1/m_j$, we have
			\begin{eqnarray*}
				R_{n+1}-R_n&=&\frac{H_{n+1,k}}{H^{\theta_j}_{n+1,j}}-\frac{H_{nk}}{H^{\theta_j}_{nj}}=\frac{H_{n+1,k}}{H^{\theta_j}_{nj}}-\frac{H_{nk}}{H^{\theta_j}_{nj}}-\left(\frac{H_{n+1,k}}{H^{\theta_j}_{n+1,j}}-\frac{H_{n+1,k}}{H^{\theta_j}_{nj}}\right)\\
				&=&\frac{A_{n+1,k}X_{n+1,k}}{H^{\theta_j}_{nj}}+(H_{nk}+A_{n+1,k}X_{n+1,k})\left(\frac{1}{(H_{nj}+A_{n+1,j}X_{n+1,j})^{\theta_j}}-\frac{1}{H^{\theta_j}_{nj}}\right).
			\end{eqnarray*}
			By the Taylor expansion of the function $f(x)=1/(x+a)^{\theta_j}$, we have 
			$$\frac{1}{(H_{nj}+A_{n+1,j}X_{n+1,j})^{\theta_j}}-\frac{1}{H^{\theta_j}_{nj}}\le -\frac{\theta_j}{H^{\theta_j+1}_{nj}}A_{n+1,j}X_{n+1,j}+\frac{\theta_j(\theta_j+1)}{2H^{\theta_j+2}_{nj}}A^2_{n+1,j}X^2_{n+1,j},$$
			Thus, define $C_3=\theta_j(\theta_j+1)$, we have
			\begin{align}\label{eqqq1}
				&\mathbb E(R_{n+1}-R_n|{\cal G}_n)\non\\
				\le&\frac{H_{nk}}{H^{\theta_j}_{nj}}\left[\frac{m_{n+1,k}N_{n+1}}{S_n}-\frac{\theta_jm_{n+1,j}N_{n+1}}{S_n}+\mathbb E\left(\frac{C_3}{2H^2_{nj}}A^2_{n+1,j}X^2_{n+1,j}|{\cal G}_n\right)\right.\non\\
				&\left.\quad\quad\quad+\mathbb E\left(\frac{C_3}{2H^2_{nj}H_{nk}}A_{n+1,k}A^2_{n+1,j}X_{n+1,k}X^2_{n+1,j}|{\cal G}_n\right)\right]\non\\
				\le&\frac{H_{nk}}{H^{\theta_j}_{nj}}\left\{\frac{N_{n+1}}{S_n}(m_{n+1,k}-\theta_jm_{n+1,j})+\frac{C^2_1C_3}{2}\frac{N_{n+1}}{H_{nj}S_n}\left[\mathbb E(A^2_{n+1,j})+\mathbb E(A_{n+1,k}A^2_{n+1,j})\right]\right\}\non\\
				=&\frac{H_{nk}}{H^{\theta_j}_{nj}}\frac{N_{n+1}}{S_n}\left[m_{n+1,k}-\theta_jm_{n+1,j}+O\left(\frac{1}{H_{nj}}\right)\right]\non
			\end{align}
			where the last euqation is due to that the replacement matrices have finite third order moments. Note that $\theta_j>m_k/m_j$, thus, by Assumption \ref{as1} and Lemma \ref{le1}, we have that for large enough $n$, $\mathbb E(R_{n+1}-R_n|{\cal G}_n)\le0$ a.s., then $R_{n}$ is eventually a supermartingale. By Theorem 2.5 of \cite{r46}, we obtain that $H_{nk}/H^{\theta_j}_{nj}$ converges a.s. to a finite random variable. Since $H_{nj}(i\in\{1,2,\dots,d\})$ tends to infinity a.s., for all $\theta_j>m_k/m_j$, and $\delta>0$, we have $H_{nk}/H^{\theta_j+\delta}_{nj}$ converges a.s. to 0. Lemma \ref{leA2} is proved due to the arbitrariness of $\delta$. 
		\end{proof}

		\begin{lemma}\label{leA3}
			We assume that Assumptions \ref{as1}-\ref{as2} and condition (\ref{eqn3}) hold. When $t=1$, we have that
			\begin{equation}\label{eq251}
				H^{-1}_{n1}=O_{a.s.}(n^{-1}),\ \ H^{-1}_{nk}=O_{a.s.}(n^{-\gamma_k})
			\end{equation}
			hold for all $k\in\{2,\dots,d\}$ and $0<\gamma_k<m_k/m_1$.
			When $t>1$, we have
			\begin{equation}\label{eq252}
				H^{-1}_{nk}=O_{a.s.}(n^{-\gamma_k}),
			\end{equation}
			hold for all $k\in\{1,\dots,d\}$ and $0<\gamma_k<m_k/m_1$.
		\end{lemma}
		\begin{proof}
			When $t=1$, then by Theorem \ref{th2}, we have that $Z_1=1$ almost surely. Thus, by Corollary \ref{co1}, we have	
			$H_{n1}/n\overset{a.s.}{\to}Nm_1\ge m_1.$
			By Lemma \ref{leA2}, for all $k\in\{2,\dots,d\}$, we have that for all $\theta_k>m_1/m_k$, $H_{n1}/H^{\theta_k}_{nk}\overset{a.s.}{\to}0$. Then,
			$$m_1\lim\limits_{n\to\infty}\frac{n}{H^{\theta_k}_{nk}}\le\lim\limits_{n\to\infty}\frac{H_{n1}}{n}\frac{n}{H^{\theta_k}_{nk}}=\lim\limits_{n\to\infty}\frac{H_{n1}}{H^{\theta}_{nk}}=0,\ \ a.s..$$
			Thus, 
			$$\lim\limits_{n\to\infty}\frac{n}{H^{\theta_k}_{nk}}=0,\ \ a.s.,\ \ {\rm which\ implies}\ \ \lim\limits_{n\to\infty}\frac{n^{1/\theta_k}}{H_{nk}}=0,\ \ a.s..$$
			Define $\gamma_k=1/\theta_k$, recall that $\theta_k>m_1/m_k$, then $0<\gamma_k<m_k/m_1$, (\ref{eq251}) is proved.
			
			When $t>1$, we define
			$$F_1=\left\{\lim\limits_{n\to\infty}\frac{H_{n1}}{n}>\frac{m_1}{t+1}\right\},\ \  F_2=\left\{\lim\limits_{n\to\infty}\frac{H_{n2}}{n}>\frac{m_1}{t+1}\right\},\ \dots,\  F_t=\left\{\lim\limits_{n\to\infty}\frac{H_{nt}}{n}>\frac{m_1}{t+1}\right\},
			$$
			Next, we will show that 
			\begin{equation}\label{eq253}
				{\mathbb P}(F_1\cup F_2\cup\dots\cup F_{t})=1.
			\end{equation}
			We only to prove ${\mathbb P}(F^c_1\cap\dots\cap F^c_t)=0$. Actually, by Corollary \ref{co1}, we have that
			$\lim\limits_{n\to\infty}S_n/n=Nm_1\ge m_1$, $H_{nk}/n(k\in\{1,\dots,t\})$ converges a.s. surely to a random variable and $H_{nj}/n(j\in\{t+1,\dots,d\})$ converges a.s. to $0$. Then, we have that
			\begin{equation*}
				m_1\le\lim\limits_{n\to\infty}\frac{S_n}{n}=\lim\limits_{n\to\infty}\frac{H_{n1}}{n}+\dots+\lim\limits_{n\to\infty}\frac{H_{nd}}{n}=\lim\limits_{n\to\infty}\frac{H_{n1}}{n}+ \dots+\lim\limits_{n\to\infty}\frac{H_{nt}}{n}
			\end{equation*}
			holds a.s.. Thus, on $F^c_1$, there exists at least one $k\in\{2,\dots,t\}$, such that
			$$\lim\limits_{n\to\infty}\frac{H_{nk}}{n}\ge\frac{m_1-\frac{1}{t+1}m_1}{t-1}=\frac{t}{t^2-1}m_1$$
			holds almost suerly. However, on $F^c_k$, we have
			\begin{equation*}
				\lim\limits_{n\to\infty}\frac{H_{nk}}{n}\le\frac{1}{t+1}m_1<\frac{t}{t^2-1}m_1,
			\end{equation*}
			which contradicts the fact that $H_{nk}/n$ converges almost surely. Thus, ${\mathbb P}(F^c_1\cap\dots\cap F^c_t)=0$, the equation (\ref{eq253}) is obtained.
			
			Then, we only need to prove that for all $k\in\{1,\dots,t\}$, on $F_k$, for all $0<\gamma_j<m_j/m_1$, $H^{-1}_{nj}=O(n^{-\gamma_j})$ holds a.s.. Note that for all $k\in\{1,\dots,t\}$ and $j\in\{1,\dots,k-1,k+1,\dots,d\}$, by Lemma \ref{leA2}, we have that for all $\theta_j>m_1/m_j$, $H_{nk}/H^{\theta_j}_{nj}\overset{a.s.}{\to}0$ holds. On $F_k$,
			$$\frac{m_1}{t+1}\cdot\lim\limits_{n\to\infty}\frac{n}{H^{\theta_j}_{nj}}\le \lim\limits_{n\to\infty}\frac{H_{nk}}{n}\frac{n}{H^{\theta_j}_{nj}}=\lim\limits_{n\to\infty}\frac{H_{nk}}{H^{\theta_j}_{nj}}=0,\ \ a.s.,$$
			which indicates
			$$\lim\limits_{n\to\infty}\frac{n^{1/\theta_j}}{H_{nj}}=0,\ \ a.s..$$
			Define $\gamma_j=1/\theta_j$, then $0<\gamma_j<m_j/m_1$. By the arbitrariness of $k$ and $j$, we obtain that (\ref{eq252}) holds. Lemma \ref{leA3} is proved.
		\end{proof}
		
		\section{Some auxiliary results}\label{secB}
		\begin{lemma}[Corollary 2.3 in \cite{r46}]\label{LeB6}
			Let $\{Y_n,n\ge1\}$ be a sequence of random varibles such that $0\le Y_n\le 1$, and $\{{\cal F}_n,n\ge1\}$ be an increasing sequence of $\sigma$-fields such that each $Y_n$ is ${\cal F}_n$-measurable. Then $\sum_n Y_n<\infty$ a.s. if and only if $\sum_n\mathbb E(Y_n|{\cal G}_{n-1})<\infty$ a.s..
		\end{lemma}
		
		\begin{lemma}[Theorem 2.17 in \cite{r46}, p.35]\label{LeB5}
			Let $\left\{S_n=\sum_{i=1}^nY_i,{\cal F}_n,n\ge1\right\}$ be a martingale and let $1\le p\le2$. Then $S_n$ converges a.s. on the set $\left\{\sum_{i=1}^\infty\mathbb E(|Y_i|^p|{\cal F}_{i-1})<\infty\right\}$.	
		\end{lemma}
		
		\begin{lemma}[An extension of Toeplitz Lemma in \cite{r46}, p.31]\label{leB2}
			Let $\{a_i,i\ge1\}$ be positive numbers and $\{Y_i,i\ge1\}$ be real random variables. If $b_n=\sum_{i=1}^{n}a_i\uparrow+\infty$, then $Y_i\overset{a.s.}{\to}X$ ensures that $b^{-1}_n\sum_{i=1}^{n}a_iY_i\overset{a.s.}{\to}X$.
		\end{lemma}
		
		\begin{lemma}[Kronecker Lemma in \cite{r46}, p.31]\label{leB4}
			Let $\{x_n,n\ge1\}$ be a sequence of real numbers such that $\sum x_n$ converges, and let $\{b_n\}$ be a monotone sequence of positive constants with $b_n\uparrow \infty$. Then $b^{-1}_n\sum_{i=1}^{n}b_ix_i\to 0$.
		\end{lemma}
		
		The next lemma extends Ces\`{a}ro's Lemma, stating that if a sequence of numbers converges, the mean of the series also converges to the same limit. This generalization to random variables implies that if a sequence of random variables converges almost surely, their mean also converges almost surely. In processes such as the urn model, the conditional distribution of a random variable is known based on current information, although the limit of the random variable may be uncertain, while the limit of its conditional expectation is determinable. Lemma \ref{leB1} addresses the limit of the scaled mean for a sequence of random variables. According to \cite{r47}, the convergence of the scaled mean of the sequence of random variables based on the limit of their conditional expectations requires specific moment conditions. If the moments of the sequence are unknown, this condition becomes hard to confirm, leading us to propose a sufficient condition for the conditional moments of the random variable.
		
		\begin{lemma}[An extension of Lemma 4.1 in \cite{r47}]\label{leB1}
			Let $1\le p\le 2$, and let $\{\alpha_j,j\ge1\}$ and $\{\beta_j,j\ge1\}$ be two sequences of strictly positive numbers with
			$$
			\beta_n \uparrow+\infty \quad \text { and } \quad \frac{1}{\beta_n} \sum_{j=1}^n \frac{1}{\alpha_j} \longrightarrow \gamma .
			$$
			Let $\left(Y_n\right)$ be a sequence of real random variables, adapted to a filtration $\mathcal{F}$. If $\mathbb E(Y_j \mid \mathcal{F}_{j-1}) \stackrel{\text { a.s. }}{\longrightarrow} Y$ for some real random variable $Y$, then 
			$$
			\frac{1}{\beta_n} \sum_{j=1}^n \frac{Y_j}{\alpha_j}\overset{a.s.}{\to}\ \gamma Y\    on\ the\ set\ \ \left\{\sum_{j=1}^\infty\frac{ \mathbb E(|Y_j|^p|{\cal F}_{j-1})}{(\alpha_j \beta_j)^p}<+\infty\right\}.
			$$
		\end{lemma}
		\begin{proof}
			Let $M_n=\sum_{j=1}^n\frac{Y_j-\mathbb E(Y_j|{\cal F}_{j-1})}{\alpha_j\beta_j}$, then $\{M_n,n\ge1\}$ is a ${\cal F}$-martingale such that 
			\begin{eqnarray*}
				&&\sum_{j=1}^\infty\mathbb E\left[\frac{|Y_j-\mathbb E(Y_j|{\cal F}_{j-1})|^p}{(\alpha_j\beta_j)^p}|{\cal F}_{j-1}\right]\\
				&\le& 2^p\sum_{j=1}^\infty\mathbb E\left[\frac{|Y_j|^p+\left(\mathbb E(|Y_j||{\cal F}_{j-1})\right)^p}{(\alpha_j\beta_j)^p}|{\cal F}_{j-1}\right]\le2^{p+1}\sum_{j=1}^\infty\frac{ \mathbb E(|Y_j|^p|{\cal F}_{j-1})}{(\alpha_j\beta_j)^p},
			\end{eqnarray*}
			thus, we obtain that $M_n$ converges almost surely on the set $\left\{\sum_{j=1}^\infty\frac{ \mathbb E(|Y_j|^p|{\cal F}_{j-1})}{(\alpha_j \beta_j)^p}<+\infty\right\}$ by Lemma \ref{LeB5}. Note that
			\begin{eqnarray*}
				\frac{1}{\beta_n}\sum_{i=1}^n\frac{Y_i}{\alpha_i}=\frac{1}{\beta_n}\sum_{i=1}^n\frac{\mathbb E(Y_i|{\cal F}_{i-1})}{\alpha_i}+\frac{1}{\beta_n}\sum_{i=1}^n\frac{Y_i-\mathbb E(Y_i|{\cal F}_{i-1})}{\alpha_i\beta_i}\beta_i.
			\end{eqnarray*}
			By Lemma \ref{leB2}, the first term on the right-hand side of the equation above converges almost surely to $\gamma Y$. And by Lemma \ref{leB4}, the second term converges almost surely to 0 on the set $\left\{\sum_{j=1}^\infty\frac{\mathbb E(|Y_j|^p|{\cal F}_{j-1})}{(\alpha_j \beta_j)^p}<+\infty\right\}$. Lemma \ref{leB1} is proved.
		\end{proof}
		
		\begin{lemma}[Lemma 3.2 in \cite{r48}]\label{leB3}
			Let $\left\{Z_n, n \geq 0\right\}$ be a random sequence measurable with respect to the filtration $\left\{{\cal F}_n\right\}$. Define
			$$
			\Lambda_n=\mathbb{E}\left(Z_{n+1}-Z_n \mid {\cal F}_n\right), \quad W_n=\mathbb{E}\big[\left(Z_{n+1}-Z_n\right)^2 \mid {\cal F}_n\big].
			$$
			Then, as $n$ tends to infinity, $Z_n$ converges to a finite value almost surely on the event $\sum_n \Lambda_n<\infty$ and $\sum_n W_n<\infty$.
		\end{lemma}
	\end{appendix}

\end{document}